\documentclass{article}

\usepackage{microtype}
\usepackage{graphicx}
\usepackage{subfigure}
\usepackage{booktabs} 

\usepackage{hyperref}

 \usepackage[accepted]{icml2024}

\usepackage{amsmath}
\usepackage{amssymb}
\usepackage{mathtools}
\usepackage{amsthm}

\newcommand{\tp}{{\scriptscriptstyle\mathsf{T}}}

\def \R{{\mathbb R}}

\let\O\undefined

\let\Re\undefined
\let\Im\undefined

\DeclareMathOperator{\O}{O}
\DeclareMathOperator{\U}{U}

\DeclareMathOperator{\V}{V}

\DeclareMathOperator{\I}{I}
\DeclareMathOperator{\GL}{GL}
\DeclareMathOperator{\SO}{SO}
\DeclareMathOperator{\Vol}{Vol}
\DeclareMathOperator{\Gr}{Gr}

\DeclareMathOperator{\SL}{SL}

\DeclareMathOperator{\tr}{tr}

\DeclareMathOperator{\SU}{SU}

\DeclareMathOperator{\Re}{Re}
\DeclareMathOperator{\Im}{Im}
\def\d{\partial}

\graphicspath{{fig/}}

\usepackage[capitalize,noabbrev]{cleveref}

\theoremstyle{plain}
\newtheorem{theorem}{Theorem}[section]
\newtheorem{proposition}[theorem]{Proposition}
\newtheorem{lemma}[theorem]{Lemma}
\newtheorem{corollary}[theorem]{Corollary}
\theoremstyle{definition}

\theoremstyle{remark}
\newtheorem{remark}[theorem]{Remark}

\usepackage[textsize=tiny]{todonotes}

\icmltitlerunning{}

\begin{document}

\twocolumn[
\icmltitle{Intrinsic Gaussian Processes on Manifolds and \\Their Accelerations by Symmetry}



\icmlsetsymbol{equal}{*}

\begin{icmlauthorlist}
\icmlauthor{Ke Ye}{comp}
\icmlauthor{Mu Niu}{yyy}
\icmlauthor{Pokman Cheung}{}
\icmlauthor{Zhenwen Dai}{sch}
\icmlauthor{Yuan Liu}{yyy}
\end{icmlauthorlist}

\icmlaffiliation{yyy}{School of Mathematics and Statistics, University of Glasgow, UK}
\icmlaffiliation{comp}{Company Name, Location, Country}
\icmlaffiliation{sch}{School of ZZZ, Institute of WWW, Location, Country}

\icmlcorrespondingauthor{Firstname1 Lastname1}{first1.last1@xxx.edu}
\icmlcorrespondingauthor{Firstname2 Lastname2}{first2.last2@www.uk}

\icmlkeywords{Machine Learning, ICML}

\vskip 0.3in
]




\begin{abstract}
Amidst the growing interest in nonparametric regression, we address a significant challenge in Gaussian processes(GP) applied to manifold-based predictors. Existing methods primarily focus on low dimensional constrained domains for heat kernel estimation, limiting their effectiveness in higher-dimensional manifolds. Our research proposes an intrinsic approach for constructing GP on general manifolds such as orthogonal groups, unitary groups, Stiefel manifolds and Grassmannian manifolds. Our methodology estimates the heat kernel by simulating Brownian motion sample paths using the exponential map, ensuring independence from the manifold’s embedding. The introduction of our strip algorithm, tailored for manifolds with extra symmetries, and the ball algorithm, designed for arbitrary manifolds, constitutes our significant contribution. Both algorithms are rigorously substantiated through theoretical proofs and numerical testing, with the strip algorithm showcasing remarkable efficiency gains over traditional methods. This intrinsic approach delivers several key advantages, including applicability to high-dimensional manifolds, eliminating the requirement for global parametrization or embedding. We demonstrate its practicality through regression case studies (torus knots and eight-dimensional projective spaces) and by developing binary classifiers for real-world datasets (gorilla skulls' planar images and diffusion tensor images). These classifiers outperform traditional methods, particularly in limited data scenarios.
\end{abstract}

\section{Introduction}
\label{sec:intro}
Due to the high dimensionality and non-linearity of data, modelling problems on Euclidean spaces is no longer always efficient and satisfied. This provides an impetus to develop new theories and methods to model problems on manifolds. There is an enormous amount of practical problems which can be naturally modelled on various manifolds and efficiently studied by statistics on them. In this work we focus on matrix manifolds which are widely used in pattern recognition, medical and biological science. For example, medical models of the human kidney can be studied by the statistics of Lie groups \cite{fletcher2003statistics}; palaeomagnetic data sets are modelled and analysed on spheres \cite{gidskehaug1976statistics}; new algorithms for problems of pedestrian detection and object categorization are proposed through the investigation of the space of positive definite matrices \cite{jayasumana2013kernel}; biological and medical images are considered as points in shape spaces \cite{bookstein2013measurement,kendall1977diffusion}; clustering algorithms on Stiefel manifolds and Grassmannian manifolds are applied to image and video recognition problems \cite{chikuse2012statistics,turaga2011statistical}. On the other hand, as one of the most prominent methods, the Gaussian process (GP) has been extensively used in statistics and machine learning on Euclidean spaces \cite{Rasmussen2004}. However, it can not be directly generalized to model data on a general manifold. A major challenge in constructing GPs on manifolds is choosing a valid covariance kernel. This is a non-trivial problem and most of the focus has been on developing covariance kernels specific to a particular manifold \cite{jayasumana2016kernels,lafferty2005diffusion, jayasumana2013kernel}. In \citet{extrinsicGP}, an extrinsic GP is proposed on manifolds by first embedding the manifolds into higher-dimensional Euclidean spaces. The well known squared exponential kernel can then be applied on the images after embedding. However, such embeddings are not always available or easy to obtain for general manifolds. Even if the embedding exists, the dimension of the embedded Euclidean space could be much higher than the original manifold. The embedding of Grassmannians \cite{extrinsicGP} is such an example. Readers can also find an example in Appendix~\ref{subsec:projective spaces} where the dimension of the manifold is $8$ but the dimension of the embedded space is $25$. 
\citet{niu2018intrinsic} proposes a class of intrinsic GP which refers to GP that employs the intrinsic Riemannian geometry of the manifold. The intrinsic GP uses heat kernels as covariance kernels. The heat kernel generalizes the popular and well-studied squared exponential kernel (also known as RBF kernel) to the Riemannian manifold, which arises from the Laplace operator and thus fully exploits the intrinsic geometry of the manifold. They utilise connections between heat kernels and transition densities of the Brownian motion(BM) on manifolds to obtain algorithms for approximating covariance kernels. \if Background and preliminaries on Gaussian processes and heat kernels are provided respectively in Subsection~\ref{subsec:background} and Appendix~\ref{appendix:heat kernel}.\fi Manifolds considered are mainly subsets of $\mathbb{R}^2$ and $\mathbb{R}^3$. 

In this paper, we present a method to simulate BM and estimate the kernel on a general Riemannian manifold. More importantly, if the manifold satisfies a symmetry condition, then our method can be accelerated by the proposed strip algorithm. Examples include classical Lie groups and their homogeneous spaces, such as orthogonal groups, Stiefel manifolds and Grassmannian manifolds. 
\if Basic facts about  Lie groups and homogeneous spaces are provided in Appendices~\ref{Liegroup} and \ref{appendix:homogeneous spaces}.\fi 

\if The approach of \cite{niu2018intrinsic} requires the parameterisation of the manifold so that the induced metric tensor can be derived from the local coordinates. The associated Laplace-Beltrami operator $\Delta_s$ of the Riemannian manifold can be defined. The Laplace-Beltrami operator is also the infinitesimal generator of the Brownian motion on the manifold. The Brownian motion on a Riemannian manifold in a local coordinate system is given by a system of stochastic differential equations (SDE) in the Ito form \cite{ hsu1988, hsu2008}. To simulate the Brownian motion sample paths, \cite{niu2018intrinsic} discretises the SDE in \cite{ hsu1988, hsu2008} using the Euler-Maruyama method (\cite{kloeden1992}). There is no issue with the SDE in its infinitesimal form. However, when it is approximated by a discretisation, the accuracy of this approximation is strongly related to the choices parametrisation. In this paper, we present a method of simulating the Brownian motion on a Riemannian manifold by exponential maps which do not require the parametrisation of the manifold. Detailed discussions are in Subsection \ref{subsec:Brownian motions on general Riem manifolds} and Section \ref{heatkernelmani} and for convenience of the reader, we also summarise corresponding algorithms in Appendix~\ref{appendix:algorithms}.

\cite{niu2018intrinsic} estimates the transition probability of the Brownian motion by counting the number of sample paths reaching the neighbourhood of the target point which is a ball with a small radius. We refer this method as the ball algorithm. It could be problematic when the dimension of the manifold is high, since the probability of the Brownian motion sample paths reaching a small ball centred on a given point could tends to zero as the dimension of the manifold increases. We propose a new approach referred as the strip algorithm in Subsection \ref{stripAlgorithm}  to estimate the transition density of the Brownian motion. It is capable of dealing with higher dimensional manifolds. The comparison between these two algorithms is given in Subsection \ref{compareBallStrip}. We test our new approach in Section~\ref{sec:examples}, which consists of two examples of function regression and two examples of classification.
\fi
\section{Intrinsic Gaussian Process on Manifolds}\label{sec:in-GP}

\subsection{Background}\label{subsec:background}
Let $M$ be a $m$-dimensional complete Riemannian manifold and let $\mathcal D = \{ (x_i,y_i), i =1,\ldots,n \}$ be the data, with $n$ the number of observations, $x_i \in M$ the predictor or location value of observation $i$ and $y_i$ the corresponding response variable. We would like to do inferences on how the output $y$ varies with the input $x$, including the prediction of the $y$-value at a new location $x_*$, which is not represented in the training dataset. Assuming Gaussian noise and a simple measurement structure, we let 
\begin{align}
y_i = f(x_i) + \epsilon_i, \  \  \    \epsilon_i \sim \mathcal N(0, \sigma_{noise}^2),  \  \  \ x_i\in M,
\end{align}
 where $\sigma_{noise}^2$ is the variance of the noise. 

We can place an Gaussian process prior for the unknown function $f:M \to \mathbb{R}$, we have 
\begin{align}
p(\text{{\bf f}} | x_1,x_2,...,x_n) = \mathcal N( \boldsymbol{0},\Sigma),
\end{align}
where {\bf f} is a vector containing the realisations of $f(\cdot)$ at the sample points $x_1,\ldots,x_n$, $f_i = f(x_i)$, and 
$\Sigma$ is the covariance matrix of these realisations induced by the intrinsic GP covariance kernel. The heat kernel $p_{t}^M(x_i,x_j)$ is used as the covariance kernel, where the time parameter $t$ of $p_{t}^M$ has the same effect as that of the length-scale parameter of the RBF kernel, controlling the rate of decay of the covariance. In particular, the entries of $\Sigma$ are obtained by evaluating the covariance kernel at each pair of locations, that is,
 \begin{align}
 \label{sk:heat}
 \Sigma_{ij} = \sigma_h^2 p_{t}^M(x_i,x_j).
 \end{align}
The hyperparameter $\sigma_h^2$ allows rescaling the heat kernel for extra flexibility.
Following standard practice for GPs, this prior distribution is updated with information in the response data to obtain a posterior distribution. The posterior distribution of $f$ evaluated at locations ${\bf X}=(x_1,...,x_n)$ has the following form:
\begin{align*}
f(x)| \mathcal{D} &\sim GP ( m_{post}, \Sigma_{post} ) \\
m_{post} &= \Sigma_{x, \bf{X} } (\Sigma_{\bf X,X} + \sigma_{noise}^2 {\bf I})^{-1} {\bf y}\\
\Sigma_{post}& = \Sigma_{x,x} - \Sigma_{x, \bf{X} } (\Sigma_{\bf X,X} + \sigma_{noise}^2 {\bf I})^{-1} \Sigma_{ \bf{X} ,x} ,
\end{align*}
where ${\bf y} = (y_1,...,y_n)$.
Let $\bf f_*$ be a vector of values of $f$ at test points which are not in the training sample. The joint distribution of $\bf f$ and $\bf f_*$ is multivariate normal. The predictive distribution $p( {\bf f_* | y} )$ is derived by marginalizing out $\bf f$. Namely, 
\small{
\[
p( {\bf f_* | y} ) = \int p(  {\bf f_* f | y}  ) d {\bf f}=\] \[ \mathcal N\left( \Sigma_{\bf f_*f} ( \Sigma_{\bf ff} + \sigma_{noise}^2 I )^{-1} {\bf y}  , 
\Sigma_{\bf f_*f_*} - ( \Sigma_{\bf ff} + \sigma_{noise}^2 I)^{-1} \Sigma_{\bf ff_*}   \right).
\] 
}

\normalsize
The key challenge for inference using intrinsic GPs is how to get the heat kernel $p_{t}^M(x_i,x_j)$. Let $\Delta_s$ be the Laplacian-Beltrami operator  on  $M$, and $\delta$ the Dirac delta function. A heat kernel of $M$ is a smooth function $p^M_t(x_i,x_j)$ on $M\times M\times \mathbb R^+$ that satisfies the heat equation:  
\begin{align}\label{eqn:heat equation 1}
\frac{\d}{\d t}p_{t}^M(x_i,x_j) &=\frac{1}{2}\Delta_s p_{t}^M(x_i,x_j).
\end{align}
$\lim_{t\rightarrow 0}p_{0}^M(x_i,x_j) =\delta(x_i,x_j), \ \ x_i, x_j \in M.$ Here the initial condition holds in the distributional sense \cite{berline2003heat}. If $M$ is the Euclidean space $\mathbb{R}^m$, the heat kernel has a closed form expression corresponding to time varying Gaussian function:
\begin{align*}
p_{t}^M({x_i},{x_j})
=\frac{1}{(2\pi t)^{d/2}}\,
  \exp\left\{-\frac{||{x_i}-{x_j}||^2}{ 2t }\right\}, \quad {x_i},{x_j} \in \mathbb R^m,
\end{align*} 
which can be seen as the scaled version of RBF kernel. Unfortunately, closed form expressions for $p_{t}^M$ do not exist for general Riemannian manifolds. 
Therefore, for most cases, one can not explicitly evaluate $p_{t}^M$ and the corresponding covariance matrices.  To revolve this issue and  bypass the need of solving the heat equation \eqref{eqn:heat equation 1} directly, we utilise the fact that the heat kernel on $M$ can be interpreted as the transition density of the Brownian motion.
We estimate the heat kernel $p_{t}^M(x_i,x_j)$ for a pair $(x_i, x_j)$ by simulating the BM on $M$ and numerically evaluate the transition density of the BM. However, unlike the method in \citet{niu2018intrinsic}, here we do not require an explicit parametrisation of the manifold. Also manifolds considered in Section~\ref{heatkernelmani} are not just simple subsets of $\mathbb R^2$ and $\mathbb R^3$. The simulation of the BM on a Riemannian manifold is discussed in Section~\ref{subsec:Brownian motions on general Riem manifolds} and the detailed algorithms are given in Section~\ref{heatkernelmani} for different types of manifolds respectively.

\subsection{ Brownian Motion on a Riemannian manifold}\label{subsec:Brownian motions on general Riem manifolds}
The goal of this subsection is to recall a method of simulating the Brownian motion on a Riemannian manifold by exponential maps, which was first proposed in \cite{mckean1960brownian} for Lie groups and later generalized to arbitrary Riemannian manifolds in \cite{gangolli1964construction}. 
Let $(M,g)$ be an $m$-dimensional Riemannian manifold with metric $g$ and let $x$ be a point on $M$. We denote by $T_x M$ the tangent space of $M$ at $x$, $W_{x}(\delta)$ a BM step with step variance $\delta$ in $T_xM$ from the origin and $\exp_x: U_x \to M$ the exponential map from an open neighborhood $U_x \subseteq T_x M \simeq \mathbb{R}^d$ around the origin. We have Algorithm~\ref{alg:simulation of BM} to numerically simulate a Brownian motion sample path on $M$.

\begin{algorithm}[tb]
   \caption{Simulation of the Brownian motion on a Riemannian manifold}
   \label{alg:simulation of BM}
\begin{algorithmic}
   \STATE  Initialize $x_0 = x$ and $B_{x_0}(0) = x_0$, $t$ is the diffusion time\; $T=t/\delta$ is the number of Brownian motion steps.
   \FOR{ $j =1, 2, \dots,T $ }
   \STATE compute $W_{x_{j-1} }( \delta )$;\ (Brownian motion sample path in $U_{x_{j-1} }$ with step variance $\delta$)
   \STATE compute $B_{x_0}(j\delta) = \exp(W_{x_{j-1}}( \delta ))$; (one-step BM sample path from $x_{j-1}$ on $M$)
   \STATE set $x_j = B_{x_0}( j \delta)$; 
   \ENDFOR
   \STATE set $\boldsymbol{B}_{x_0} = \{ B_{x_0}(0), B_{x_0}(\delta), B_{x_0}(2\delta), \dots, B_{x_0}(t)  \}$; (BM sample path from $x_0$ on $M$)
\end{algorithmic}
\end{algorithm}
The following fact ensures that $\boldsymbol{B}_{x_0}$ obtained in Algorithm \ref{alg:simulation of BM} is indeed an approximation of a BM sample path on $(M,g)$. 
\begin{theorem}\cite{gangolli1964construction}\label{thm:BM on general manifolds}
The path $\boldsymbol{B}_{x_0} $ converges to a Brownian sample path on $M$ with probability $1$, as $\delta \to 0$.
\end{theorem}

We remark that in practice the main difficulty of applying Algorithm \ref{alg:simulation of BM} to simulate BM sample paths is that in general, the exponential map on $(M,g)$ is not explicitly known. However, we will see in the sequel that for the Riemannian manifolds considered in this paper such as Lie groups and their homogeneous spaces, Algorithm \ref{alg:simulation of BM} is a convenient and efficient way to simulate BM sample paths. Detailed algorithms of simulating the BM on different types of manifolds are given in Section~\ref{heatkernelmani}. In order to compare Algorithm~\ref{alg:simulation of BM} with the method using stochastic differential equation in \citet{niu2018intrinsic}, we consider the example of the standard sphere parametrised by latitude $\phi$ and longitude $\theta$ in details. This coordinate system cannot cover the north and south poles -- they are the singularities. In this coordinate system, the BM in \citet{niu2018intrinsic} is described by 
$(d\phi, d\theta )  = ( -\frac{\tan(\phi) }{2}  dt + dB_{\theta}, \frac{1}{\cos(\phi)} dB_{\phi} ).$
Note that near the poles the drift velocity $( - \tan(\phi) /2  , 0 )$ becomes large and points away from the poles (huge repulsive drift); There is no issue with the stochastic differential equation in its infinitesimal form. However, when it is approximated by a discretisation, near the poles the drift term may become too large for the simulated step to be a good approximation of the actual BM. However in our approach, the simulation of BM does not require the parametrisation of the space and hence we do not have such a problem at the poles.

\section{Estimate the heat kernel from BM paths}
To explore the connection between the heat kernel and the BM on a manifold, we let $B_{x_0}(t)$ be the BM on $M$ starting from $x_0 = B_{x_0}(0)$. The transition probability of $B_{x_0}(t) \in D \subseteq M$ at time $t$, for any Borel set $D$, is given by 
\begin{align}
\label{BMprob}
\mathbb{P}\big[B_{x_0}(t)\in D\,|\,B_{x_0}(0)=x_0\big] = \int_D p^M_t(x_0,x)dx,
\end{align}
where the integral is defined with respect to the volume form of $M$. In \cite{niu2018intrinsic}, authors estimate the heat kernel on a given manifold via approximating the integral in \eqref{BMprob} by simulating the BM sample paths and numerically evaluating the transition probability. In the rest of this section, we will discuss how to efficiently simulate BM sample paths on manifolds using two methods of numerical differential geometry.

\begin{figure*}[t]
\begin{center}
 \subfigure[ball algorithm]{ \label{fig:ball} \includegraphics[width=0.32\textwidth,height=0.32\textwidth]{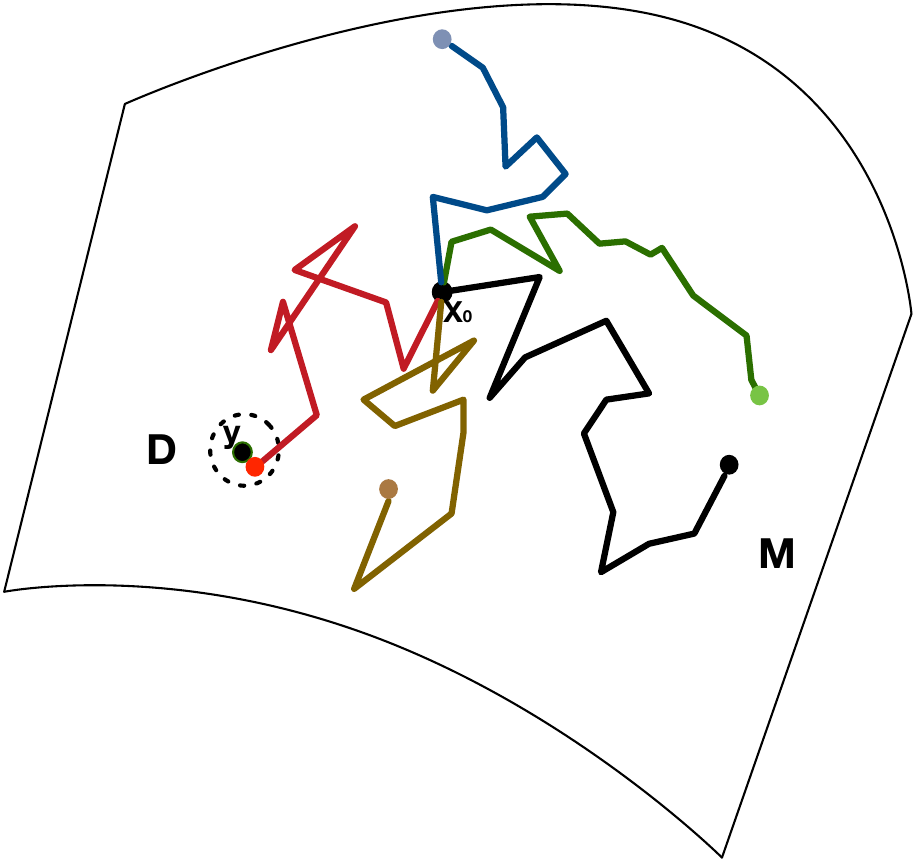} } 
    \subfigure[strip algorithm]{ \label{fig:strip} \includegraphics[width=0.32\textwidth,height=0.32\textwidth]{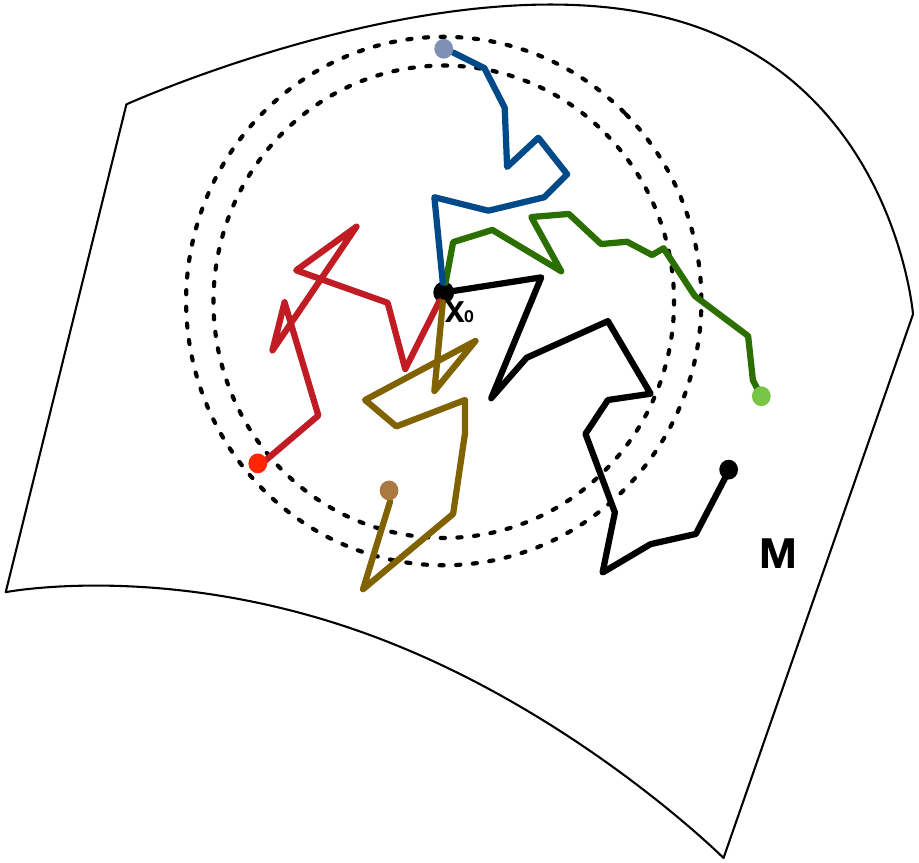} }
\caption{ \label{fig:examples} Illustrative examples of the Brownian motion on manifold $M$. Five independent BM paths from time $0$ to $t$ represented by the solid coloured lines in both \ref{fig:ball} and \ref{fig:strip}. $x_0$ is the starting point of BM paths. In \ref{fig:ball} only one sample path (red) reaches $D$ at time $t$, so the transition probability is $1/5$. In \ref{fig:strip} there are two paths (red and blue) reach the strip at time $t$ so the transition probability is $2/5$.}   
\end{center}
\end{figure*}

\subsection{Ball algorithm}
We first recall from \cite{niu2018intrinsic} that the transition probability in \eqref{BMprob} can be approximated by counting the number of sample paths reaching the neighbourhood of the target point. We refer to it as the ball method. Let $\{ B_x(t): t>0 \}$ be the BM on $M$ with starting point $B_x(0)=x$, $x\in M$ and let $N$ be the number of simulated sample paths. For $t>0$ and $y\in M$ the probability of $B_x(t)$ in a small neighbourhood $D$ of $y$ can be estimated by counting how many BM paths reach $D$ at time $t$. An illustrative diagram is shown in Figure \ref{fig:ball}. The transition density is approximated by
\begin{equation}\label{eqn:heat kernel estimate}
p_t^M(x,y) \approx \frac{1}{\Vol(D)} \frac{k}{N},
\end{equation}
where $\Vol(D)$ is the volume of $D$ and $k$ is the number of Brownian motion sample paths falling into $D$ at time $t$. Based on \eqref{eqn:heat kernel estimate} and Algorithm~\ref{alg:simulation of BM}, we have Algorithm~\ref{alg:estimation of heat kernel} to estimate the heat kernel $p_t^M(x,y)$.

\begin{algorithm}[tb]
   \caption{estimate of the heat kernel by ball Algorithm}
   \label{alg:estimation of heat kernel}
\begin{algorithmic}
\STATE given $x,y\in M$, $t > 0$, $m, N\in \mathbb{N}$, $\epsilon > 0$, compute $V = \Vol (\lbrace z\in M: d(z,y) \le \epsilon \rbrace)$;
\STATE set $k = 0$; 
\FOR{ $i =1, \dots, N$ }
\STATE sample a Brownian motion $B_x(\tau)$ for $\tau \in [0,t]$ starting from $x$ by Algorithm~\ref{alg:simulation of BM};
 \STATE set $z = B_x (t)$; 
  \IF{$d(y,z) < \epsilon$}
  \STATE set $k = k + 1$;  
  \ENDIF
  \ENDFOR
 \STATE set $p  = \frac{k}{NV}$. $\{$an estimate of $p_t^M(x,y)$ $\}$;
\end{algorithmic}
\end{algorithm}

\subsection{Strip algorithm} \label{stripAlgorithm}
We present in this subsection an accelerated algorithm for an estimate of the heat kernel on a given manifold. We call it the strip algorithm. Suppose that $M$ is a Riemannian manifold of dimension $m$. We denote by $d(x,y)$ the geodesic distance between $x,y\in M$, i.e., the minimum of lengths of piecewise smooth curves on $M$ connecting $x$ and $y$. We define for each $x\in M$, $d_0 \ge 0,\epsilon >0$ the set
\begin{align*}
S_x(d_0,\epsilon) :=\left \{
z\in M: \lvert d(x,z) -  d_0 \rvert < \epsilon 
\right \}
\end{align*}
We also denote by $S_x(d_0)$ the set consisting of all $z\in M$ such that $d(x,z) = d_0$. From now on, we call $S_x(d_0,\epsilon)$ the $\epsilon$-strip of $S_x(d_0)$. It is tempting to think $S_x(d_0,\epsilon)$ as the tube of $S_{x}(d_0)$ of radius $\epsilon$. This is not always the case, but it is true with a minor assumption on $M$, which is obviously satisfied by manifolds we considered in this paper. We refer readers to Lemma~\ref{lemma:tube} for the description of this minor assumption. Moreover, we prove in Lemma~\ref{lemma:submanifold} that $S_{x}(d_0)$ is a closed submanifold of $M$ for most choices of $d_0$. Hence we may apply Theorem~\ref{thm:tube formula} to estimate the volume of the tube $S_x(d_0,\epsilon)$. The aforementioned technical results are all recorded in Appendix~\ref{appendix: tubes}.

We assume that the heat kernel $p^M_t(x,y)$ is a function of distance at some $x\in M$, i.e., 
\begin{equation}\label{eqn:distance dependence}
p_t^M(x,y_1) = p_t^M(x,y_2),\quad \text{whenever}~d(x,y_1) = d(x,y_2).
\end{equation}
Manifolds satisfying \eqref{eqn:distance dependence} include Euclidean spaces, spheres, projective spaces and their quotients. For such a manifold, we have the following analogue of \eqref{eqn:heat kernel estimate}:
\begin{equation}\label{eqn:heat kernel estimate acceleration}
p^M_t(x,y) \approx \frac{1}{\Vol(S_x(d_0,\epsilon))} \frac{k}{N},
\end{equation}
where $d_0 = d(x,y)$, $N$ is the total number of sampling BM paths starting from $x$ on $M$ and $k$ is the number of BM paths falling into $S_x(d_0,\epsilon)$ at time $t$. From \eqref{eqn:heat kernel estimate acceleration} we obtain Algorithm~\ref{alg:acc estimation of heat kernel} for the estimate of the heat kernel on a manifold satisfying property \eqref{eqn:distance dependence}.
Once we have the point estimates of $p^M_t(x,y)$ for some pairs of $(x, y)$, we can use some standard interpolation methods to learn $p^M_t(x,y)$ as a function of distance.
\begin{algorithm}[tb]
   \caption{estimation of heat kernel on a manifold by strip algorithm}
   \label{alg:acc estimation of heat kernel}
\begin{algorithmic}
\STATE given $x\in M$, $d_0> 0 $, $t > 0$, $m, N\in \mathbb{N}$, $\epsilon > 0$, compute $V = \Vol(S_x(d_0,\epsilon))$;
\STATE set $k = 0$; 
\FOR{ $i =1, \dots, N$ }
\STATE  sample a Brownian path $B_x(\tau)$ for $\tau \in [0,t]$ starting from $x$ by Algorithm~\ref{alg:simulation of BM}; 
 \STATE set $z = B_x (t)$; 
  \IF{ If $d(x,z) < d_0 + \epsilon$ and $d_0 - \epsilon < d(x,z)$}
  \STATE set $k = k + 1$;  
  \ENDIF
  \ENDFOR
 \STATE set $p  = \frac{k}{NV}$. $\{$an estimate of $p_t^M(x,y)$ for any $(x,y)$ with $d(x,y) = d_0$ $\}$;
\end{algorithmic}
\end{algorithm}
Mathematically, \eqref{eqn:heat kernel estimate acceleration} can be formulated as the following result which in turn also validates Algorithm~\ref{alg:acc estimation of heat kernel}.
\begin{theorem}
We denote by $\hat{p}^M_t(x,y)$ the estimator $\frac{1}{\Vol(S_x(d_0,\epsilon))} \frac{k}{N}$ in Algorithm~\ref{alg:acc estimation of heat kernel}, where $\epsilon$ is the radius of the strip, $N$ is the number of BM sample paths and $d_0$ is the geodesic distance between $x$ and $y$. Then $\hat{p}^M_t(x,y)$ is asymptotically unbiased and consistent, i.e., 
\[
\lim_{\epsilon \to 0} \operatorname{E}( \hat{p}^M_t(x,y) ) = {p}^M_t(x,y), \quad 
    \lim_{\frac{\epsilon^m}{N} \to 0} \operatorname{Var}( \hat{p}^M_t(x,y) ) = 0,
\]
where $m$ is the dimension of the manifold $M$.
\end{theorem}
The proof of Theorem 3.1 is provided in Appendix E.

The idea behind the strip method is similar to that of the ball method: If we treat the strip as a collection of many small balls, then the transition probability for the strip would approximately be the sum of the transition probability for these balls. Moreover, the transition density is calculated as the ratio of the transition probability and the strip volume. An illustrative diagram is shown in Figure \ref{fig:strip}. For a fixed $x\in M$, if $p^M_t(x,y)$ does not satisfy \eqref{eqn:distance dependence}, then Algorithm~\ref{alg:acc estimation of heat kernel} estimates the average density of $p^M_t(x,y)$ on $S_x(d_0,\epsilon)$. To be more precise, we have 
\begin{align*}
\frac{1}{\Vol(S_x(d_0,\epsilon))} \int_{S_x(d_0,\epsilon)} p^M_t(x,y) dy \approx  \frac{1}{\Vol(S_x(d_0,\epsilon))} \frac{k}{N},
\end{align*}
where $k,N$ and $\Vol(S_x(d_0,\epsilon))$ are the same as those appeared in \eqref{eqn:heat kernel estimate acceleration}.

To conclude this subsection, we remark that Algorithm~\ref{alg:estimation of heat kernel} requires the computation of the volume of the strip $S_x(d_0,\epsilon)$. For readers' convenience, we record formulae for the volume of a tube on a Riemannian manifold in Appendix~\ref{appendix: tubes}.

\subsection{Comparison of Ball Algorithm and Strip Algorithm} \label{compareBallStrip}
It is clear that Algorithm~\ref{alg:estimation of heat kernel} is applicable to a more general class of manifolds than Algorithm~\ref{alg:acc estimation of heat kernel}, since the latter can only be applied to $M$ where $p_t^M(x,y)$ only depends on $d(x,y)$ if we fix $t$ and $x\in M$. 
\if The closed formulas of heat kernel are only available for Euclidean space and real hyperbolic space. Here the ``closed formula" means analytic expression. The power series expression of the heat kernel may exist for some manifolds such as sphere.
The ball algorithm is applicable for all manifolds considered in this paper. The strip algorithm is only applicable for some in the list.\fi However, when both algorithms are applicable, intuitively Algorithm~\ref{alg:acc estimation of heat kernel} is more efficient than Algorithm~\ref{alg:estimation of heat kernel}. \if in the sense that Algorithm~\ref{alg:acc estimation of heat kernel} requires fewer Brownian motion sample paths to estimate the heat kernel. In other words, a lot more sample paths need to be simulated to reach the `ball' than the `strip'.\fi For concrete examples, we compare in Figure~\ref{fig:algorithm comparison} the two algorithms on $\mathbb{R}^n$ for $n= 1,2,3$ respectively. By using the same number of Brownian motion sample paths, the strip algorithm outperforms the ball algorithm as the dimension of $M$ increases. Theorem~\ref{prop:compare algorithms} below explains why the strip method is much more efficient.

\begin{figure*}[h!]
\begin{center}
 \subfigure[$\mathbb{R}^1$]{ \label{ex:R1} \includegraphics[width=0.31\textwidth,height=0.31\textwidth]{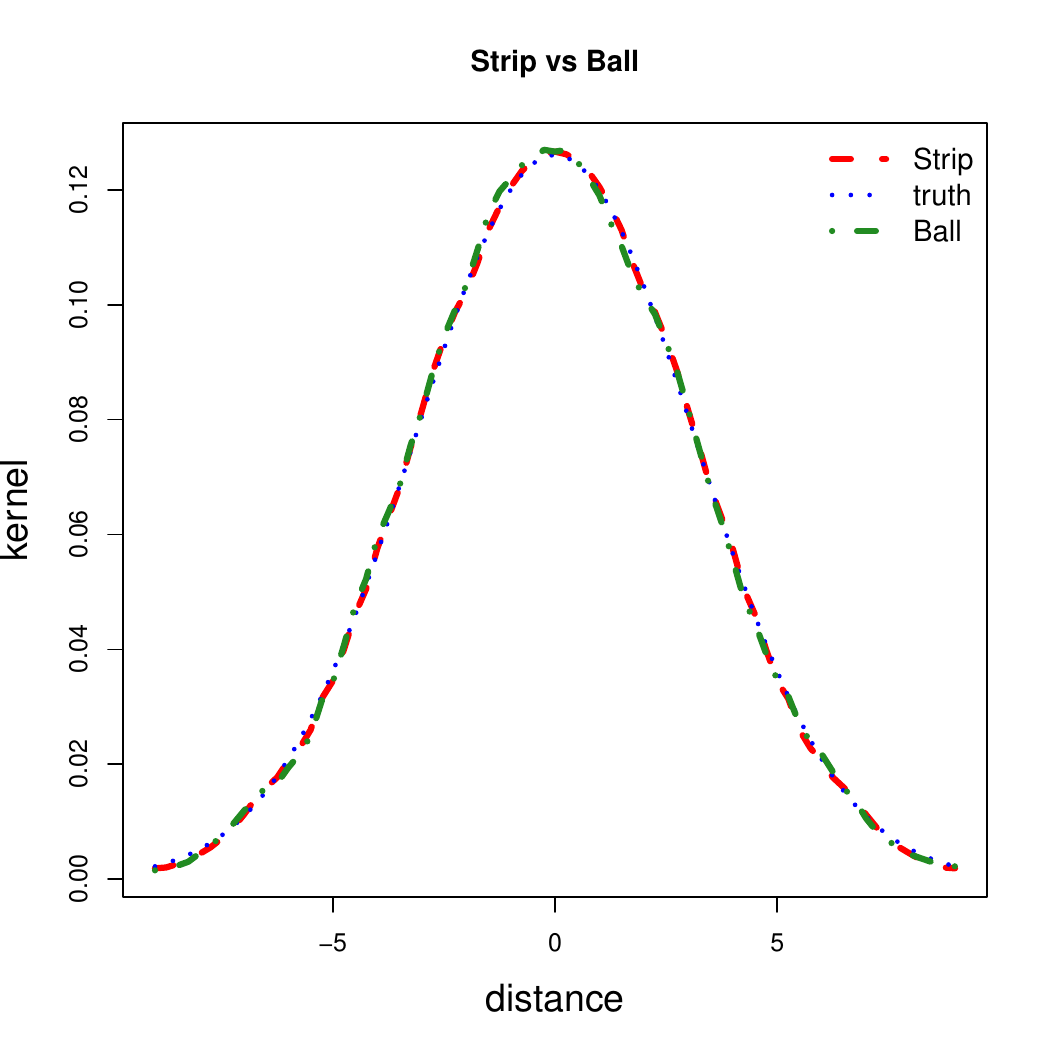} } 
 \subfigure[$\mathbb{R}^2$]{ \label{ex:R2} \includegraphics[width=0.31\textwidth,height=0.31\textwidth]{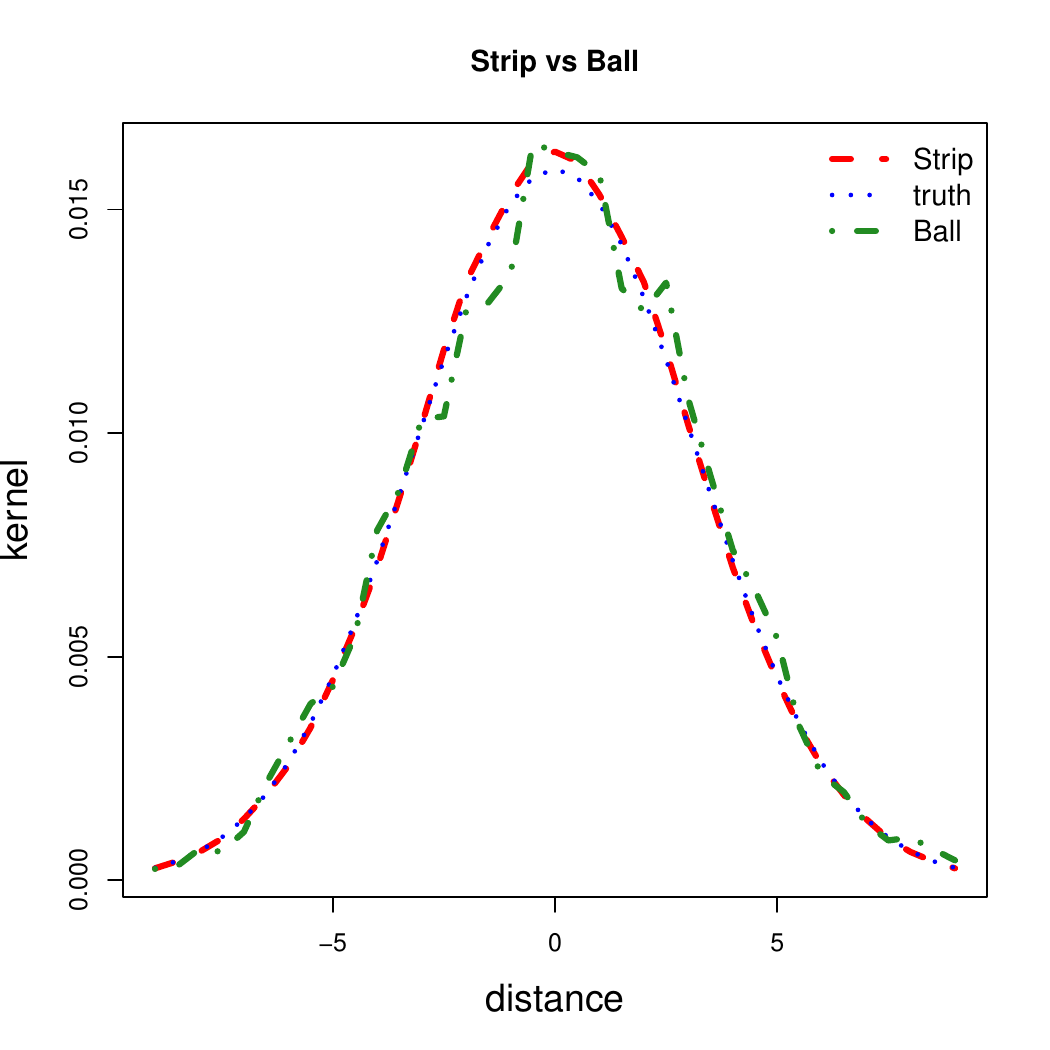} }
  \subfigure[$\mathbb{R}^3$]{ \label{ex:R3} \includegraphics[width=0.31\textwidth,height=0.31\textwidth]{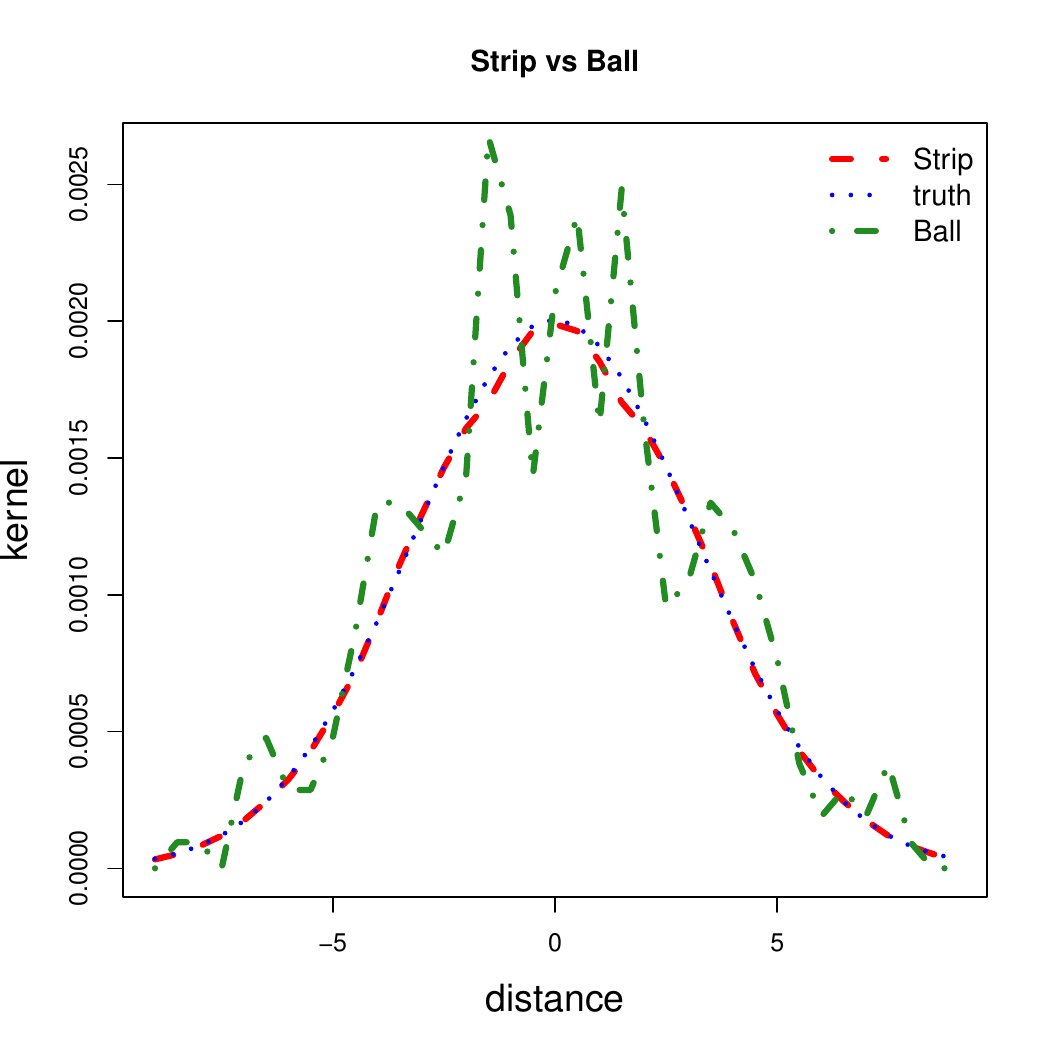} }
    \caption{  \label{fig:algorithm comparison} Comparison of the estimate of heat kernel on $\mathbb{R}^1$, $\mathbb{R}^2$ and $\mathbb{R}^3$ using strip and ball algorithms. The number of Brownian motion sample paths are 20000 for all three cases. The same window size and strip width are used. The true kernel values are plotted in dotted blue line while the estimations from ball algorithm are in green dot-dash line and strip method are in red dash line. }   
  \end{center}
\end{figure*}

\begin{theorem}\label{prop:compare algorithms}
Let $M$ be a Riemannian manifold of dimension $m$ and let $x_0,y_0\in M$ be fixed points with distance $d_0 = d(x_0,y_0)$. Suppose that $p_t^M(x_0,y)$ satisfies \eqref{eqn:distance dependence}. Then for any 
$a  >0$ there exists some $ \epsilon_0 > 0$ such that for each $0 < \epsilon \le \epsilon_0$, 
\begin{align}\label{eqn:compare algorithms}
\frac{\mathbb{P}\{B_{x_0}(t)\in S_{x_0}(d_0,\epsilon) \}}{ \mathbb{P}\{B_{x_0}(t) \in D(y_0,\epsilon)\} } &\approx \frac{\Vol(S_{x_0}(d_0,\epsilon))}{\Vol(D(y_0,\epsilon))}\\ &\ge
 \frac{\pi^{\frac{-m+1}{2}} (\frac{1}{2} m) !}{(1+a)^2(\frac{1}{2}) !} \epsilon^{-m+1}, 
\end{align}
where $B_{x_0}(t)$ is the point at time $t$ on a Brownian motion sample path starting from $x_0$ and $D(y_0,\epsilon)$ is the set of points on $M$ whose distance to $y_0$ is at most $\epsilon$.
\end{theorem}

\begin{proof}
Clearly \eqref{eqn:compare algorithms} follows from assumption \eqref{eqn:distance dependence} and the relation between the heat kernel and Brownian motion:
\[
\mathbb{P} \{B_x(t) \in D\} = \int_D p^M_t(x,y) dy.
\]
The inequality is a direct consequence of Theorem~\ref{thm:tube formula}.
\end{proof}

We remark that the lower bound in \eqref{eqn:compare algorithms} is not sharp. For example, in one dimensional case we have 
\[
\frac{\mathbb{P}\{B_{x_0}(t)\in S_{x_0}(d_0,\epsilon) \}}{ \mathbb{P}\{B_{x_0}(t) \in D(y_0,\epsilon)\} } \approx \frac{\Vol(S_{x_0}(d_0,\epsilon))}{\Vol(D(y_0,\epsilon))} \ge \frac{1}{(1+a)^2}.
\]
If $M = \mathbb{R}$ or $\mathbb{S}^1$, then it is straightforward to verify that 
\[
\frac{\mathbb{P}\{B_{x_0}(t)\in S_{x_0}(d_0,\epsilon) \}}{ \mathbb{P}\{B_{x_0}(t) \in D(y_0,\epsilon)\} }  = 2 > \frac{1}{(1+a )^2}.
\]
However, if $\epsilon$ is small, the lower bound in \eqref{eqn:compare algorithms} is already an exponential function in $m$. Hence we can conclude that Algorithm~\ref{alg:acc estimation of heat kernel} is exponentially more efficient than Algorithm~\ref{alg:estimation of heat kernel}. In other words, \eqref{eqn:compare algorithms} implies that if we fix the number of sampling Brownian paths, then the number of paths falling into the strip $S_{x_0}(d_0,\epsilon)$ is much more than the number of paths falling into the ball $D(y_0,\epsilon)$.


\section{BM on matrix manifolds}\label{heatkernelmani}
Although Algorithm~\ref{alg:simulation of BM} is applicable to an arbitrary Riemannian manifold, we may obtain a more efficient algorithm by further exploring the geometric structure of the underlying Riemannian manifold. This is the case for matrix manifolds such as Lie groups, Stiefel manifolds and Grassmannian manifolds. Readers who are not familiar with those manifolds can find basic facts about matrix Lie groups in Appendix~\ref{Liegroup} and about Stiefel manifolds and Grassmannian manifolds in Appendix~\ref{appendix:homogeneous spaces}. For each aforementioned manifold, we obtain one efficient algorithm. Due to the page limit, we record these algorithms in Appendix~\ref{appendix:algorithms} (cf. Algorithms~\ref{alg:simulation of BM on O(n)}--\ref{alg:Brownian paths on projective spaces}) and provide details of the derivation of these algorithms in Appendix~\ref{appendix:GP on matrix manifold}. Moreover, with these Algorithms in hand, we can apply Algorithms~\ref{alg:estimation of heat kernel} and \ref{alg:acc estimation of heat kernel} to efficiently estimate heat kernels on these manifolds.
\section{Examples}\label{sec:examples}
We recall that by the ball method (Algorithm~\ref{alg:estimation of heat kernel}), one can estimate the heat kernel on any Riemannian manifold. Moreover, the estimation can be performed even more effectively by the strip method (Algorithm~\ref{alg:acc estimation of heat kernel}) if the manifold satisfies condition \eqref{eqn:distance dependence}. Now that the estimate of the heat kernel is obtained, it is straightforward to construct the intrinsic Gaussian process as introduced in Subsection~\ref{subsec:background}. In this work, we provide four examples to which the strip method applies. 

We conduct simulation studies for regression problems on torus knots and high-dimensional projective spaces and develop a binary classifier with real datasets from biology and medical image processing. In the sequel, the corresponding intrinsic Gaussian process is called the strip intrinsic Gaussian process (SiGP). The performance of the SiGP is compared with the extrinsic Gaussian process(exGP) \cite{extrinsicGP}, which is the euclidean Gaussian process using RBF kernel with embedding. Due to the page
limit, the results for projective spaces are presented in Appendix G.




\subsection{Torus knots} \label{subsec:torus knots}
The circle has a family of embeddings into $\mathbb{R}^3$, whose images are torus knots indexed by a pair of coprime positive integers $\left( p,q \right)$. The explicit embedding associated to $\left( p,q \right)$ can be found in standard textbooks in knots such as \cite{milnor1968singular,rolfsen2003knots,murasugi2007knot}. In particular, we considered regression problems on three types of torus knots $\left( p=2,q=3 \right)$, $\left( p=4,q=3 \right)$ and $\left( p=9,q=8 \right)$. As submanifolds of $\mathbb{R}^3$, torus knots of different types are twisted in dramatically different ways. For example, readers can find pictures of torus knots of types mentioned above in Figure \ref{fig:algorithm comparison}. We will see in Figure \ref{fig:algorithm comparison} that our intrinsic method is not affected by these twists while the extrinsic method does depend on them. The regression function on the torus knots is defined as
\begin{eqnarray} \label{regression eqn}
f(X) =X^{\tp} {\mathcal M} X + \epsilon,
\end{eqnarray}
where $\mathcal M$ is a fixed $2 \times 2$ real positive definite matrix, $X$ is a $2$-dimensional unit norm real vector and $\epsilon$ is a i.i.d noise. The true function on the three type of torus knots are plotted in Figure \ref{ex:knot23_true} \ref{ex:43R1} and \ref{ex:98R1}.

The heat kernel on Torus note is estimated by Algorithm~\ref{alg:Brownian paths on spheres}. The SiGP approach is compared with the extrinsic approach by embedding torus knots in $\mathbb R^3$. On the one hand, the predictive means of the GP with embedding which is equivalent to using the Euclidean GP with RBF kernel in $R^3$ is shown in Figure \ref{ex:R2}, \ref{ex:43R2} and \ref{ex:98R2}. It is clear that the prediction from the GP with embedding does not agree with the truth when the crossing number increases.
On the other hand, the predictive means of SiGP in Figure \ref{ex:R3}, \ref{ex:43R3} and \ref{ex:98R3}, which uses Algorithm~\ref{alg:estimation of heat kernel} and Algorithm~\ref{alg:Brownian paths on spheres}, recovers the true function very well. The numerical comparison is shown in Table \ref{sample-table}. The root mean-squared errors (RMSE) are computed for the true function and the predictive means over 10 datasets. The mean and standard deviation (values in brackets) of RMSEs are listed for both methods in different Torus knots. The prediction of SiGP is significantly better than the extrinsic GP.

\begin{table}[t]
\caption{Comparison of the RMSE of predictive means of two methods on Torus knots. Values in parentheses show the standard deviation.}
\label{sample-table}
\vskip 0.15in
\begin{center}
\begin{small}
\begin{sc}
\begin{tabular}{lcccr}
\hline
      & exGP & SiGP \\
\hline
knot(2,3)   & 0.103(0.014)& 0.072 (0.01)&  \\
knot(4,3) & 1.78(0.58)& 0.072 (0.01) & \\
knot(9,8)  & 2.29(0.75)& 0.072 (0.01) &  \\ 
\hline
\end{tabular}
\end{sc}
\end{small}
\end{center}
\vskip -0.1in
\end{table}

\begin{figure*}[h!]
 \subfigure[$knot(2,3)$ Truth]{ \label{ex:knot23_true}  \includegraphics[width=0.25\textwidth,height=0.25\textwidth]{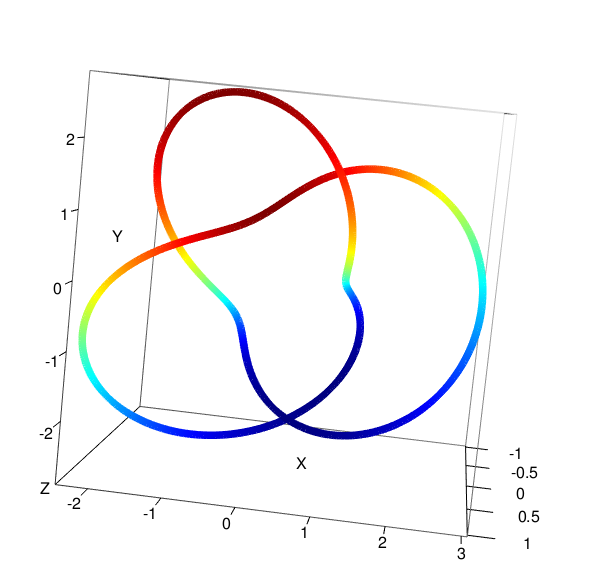} } 
 \subfigure[$knot(2,3)$ RBF]{\label{ex:R2} \includegraphics[width=0.25\textwidth,height=0.25\textwidth]{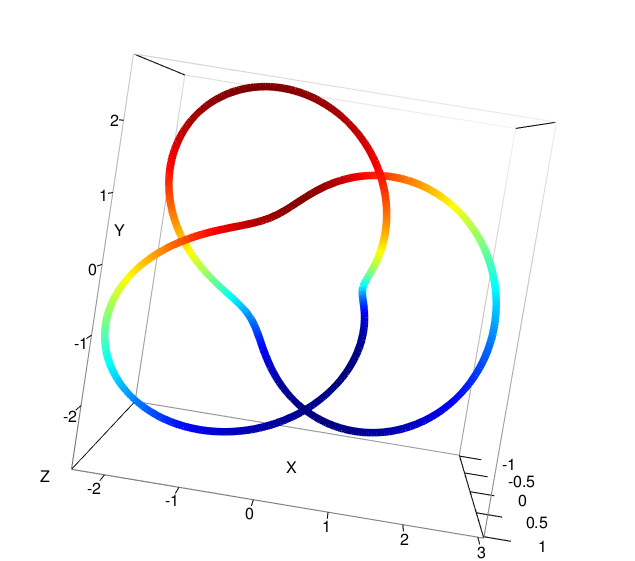} }
  \subfigure[$knot(2,3)$ Intrinsic]{ \label{ex:R3} \includegraphics[width=0.25\textwidth,height=0.25\textwidth]{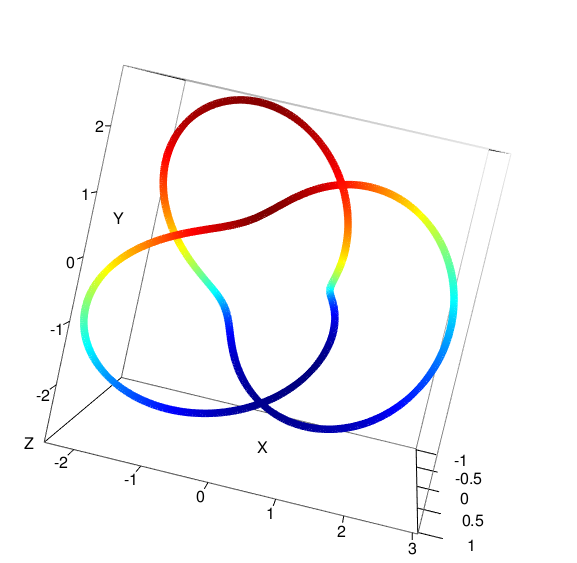} }
   \subfigure[$knot(4,3)$ Truth]{ \label{ex:43R1}  \includegraphics[width=0.25\textwidth,height=0.25\textwidth]{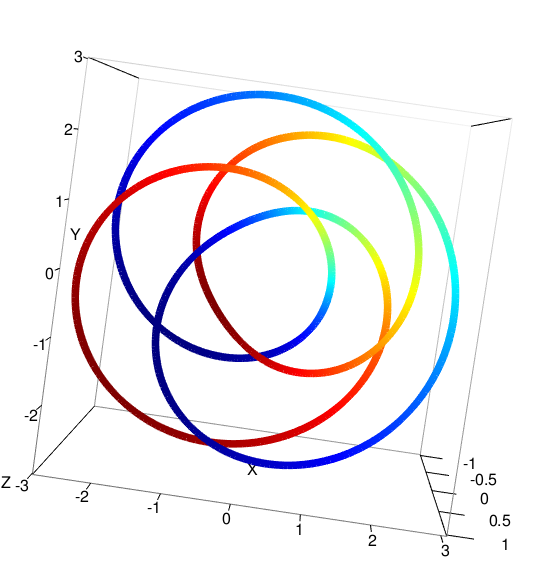} } 
 \subfigure[$knot(4,3)$ RBF]{\label{ex:43R2} \includegraphics[width=0.25\textwidth,height=0.25\textwidth]{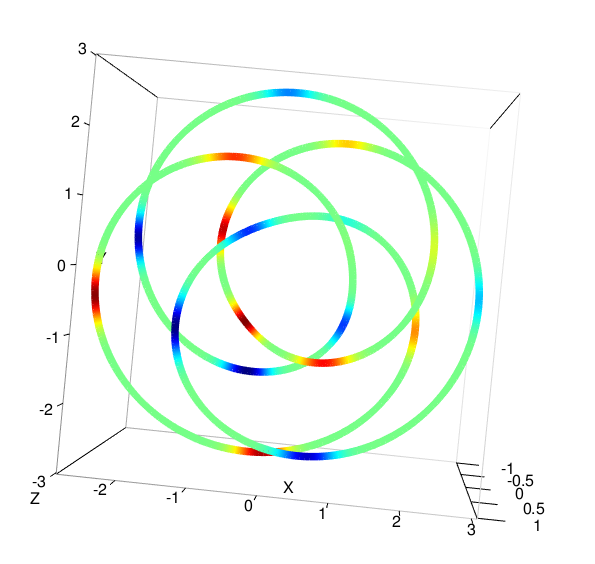} }
  \subfigure[$knot(4,3)$ Intrinsic]{ \label{ex:43R3} \includegraphics[width=0.25\textwidth,height=0.25\textwidth]{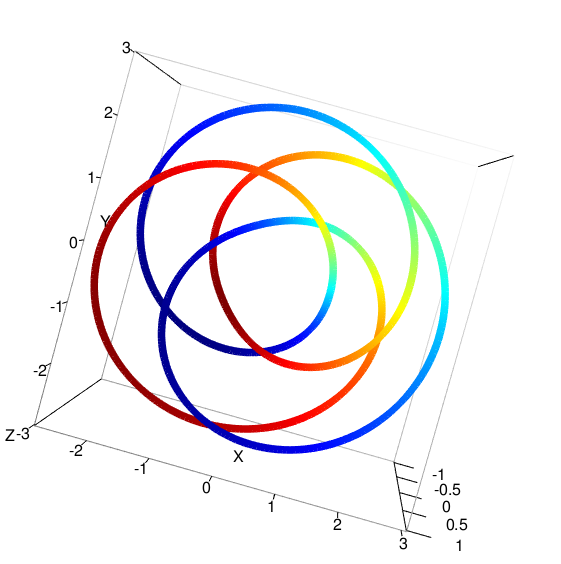} }
     \subfigure[$knot(9,8)$ Truth]{ \label{ex:98R1}  \includegraphics[width=0.3\textwidth,height=0.3\textwidth]{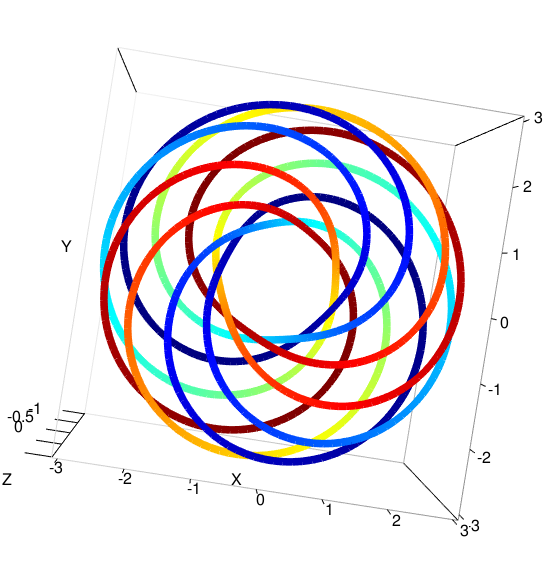} } 
 \subfigure[$knot(9,8)$ RBF]{\label{ex:98R2} \includegraphics[width=0.3\textwidth,height=0.3\textwidth]{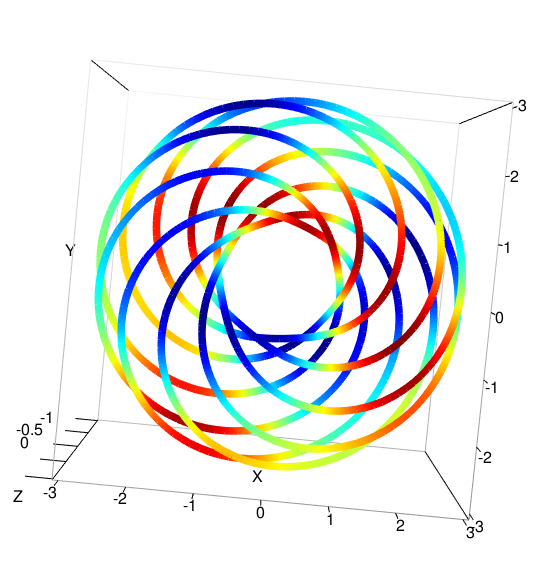} }
  \subfigure[$knot(9,8)$ Intrinsic]{ \label{ex:98R3} \includegraphics[width=0.3\textwidth,height=0.3\textwidth]{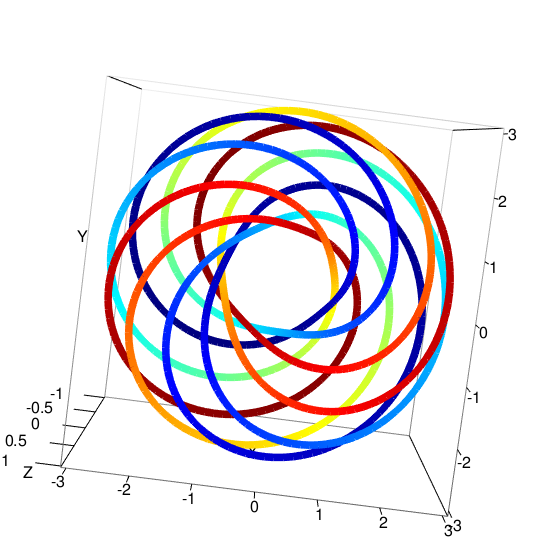} }
    \caption{  \label{fig:algorithm comparison} Comparison of the truth and prediction using extrinsic GP with RBF kernel embedded in $R^3$ and SiGP. The true value of the regression function is plotted in colour. }   
\end{figure*}


\subsection{Examples with real data sets}
\subsubsection{Classification of gorilla skulls}
The data set discussed in this subsection is contained in \citet{dryden2016}, which consists of $59$ planar images of gorilla skulls with landmarks. 
There are $29$ images for male gorillas and $30$ images for female gorillas. 
The task is to classify the genders of gorillas based on the location of the four landmarks on images.

We apply our SiGP approach to this classification problem. We notice that in this classification problem, only shapes of skulls are concerned. This indicates that our classification results should be invariant to translations and rotations of these images. The manifold we considered here is the projective space $\mathbb{P}_{\mathbb{C}}^{2}$. The heat kernel of $\mathbb{P}_{\mathbb{C}}^{2}$ is estimated by Algorithm \ref{alg:Brownian paths on projective spaces}. Detailed discussion is shown in Appendix G. 
We define $y_i \in \{ 0 ,1 \}$ where 0 represents female and 1 for male. The intrinsic GP classifier is defined as:
\begin{align} \label{eq:gorilla}
y_k &\sim Bernoulli(\pi_k),  \nonumber 
\ \ \pi_k = \Phi(f({[\overline{x}_k]})),  \nonumber \\
f(\cdot) &\sim Intrinsic \ GP(0 , p^{ \mathbb{P}_{\mathbb{C}}^{2}} )
\end{align}
where $\Phi$ is the standard normal cdf and $p^{ \mathbb{P}_{\mathbb{C}}^{2}}$ is the estimated heat kernel of the projective space $\mathbb{P}_{\mathbb{C}}^{2}$. The likelihood is approximated using expectation propogation as in \cite{hernandez16}.

From the 59 skull images, we create 10 random train-test slits with 50 images for training and 9 images for testing. The performance is measured in terms of cross entropy and the area under the Receiver Operating Characteristics curve (AUC-ROC).
The mean and standard deviation of the two metrics across 10 splits are summarised in Table \ref{tab:p2class-summary}. 
%
%
The results of SiGP are compared with a extrinsic GP classifier, SVM and LR(logistic regression). 
The RBF kernel with the Euclidean distance on $\mathbb{R}^8$ of the landmarks are used in the extrinsic GP classifier. 
SVM and LR have higher values of cross entropy than SiGP as a result of the classification probabilities being close to 0.5. 
For complex manifolds such as projective spaces, the naive representation of the data without properly incorporating the underlying geometry, lead to a  posterior estimate of the latent function that is close to the prior mean. The AUC-ROC scores also indicate that the SiGP almost perfectly identifies the gender.



\begin{table}[t]
\caption{Gorilla gender classification based on gorilla skull landmarks. The values are the means of ten random train-test splits and the values in parentheses are the standard deviation. The 1st row shows the cross entropy. The 2nd row shows the AUC-ROC. }
\label{tab:p2class-summary}
\vskip 0.15in
\begin{center}
\begin{small}
\begin{sc}
\begin{tabular}{lccccccccr}
\hline
      & SiGP  & exGP  & SVM &LR\\
\hline
  & 0.15(0.05)& 0.58(0.21) &0.70(0.004)&3.8(1.2)  \\
 &  0.99(0.01) & 0.83(0.14) &0.52(0.06)&0.92(0.07)  \\
\hline
\end{tabular}
\end{sc}
\end{small}
\end{center}
\vskip -0.1in
\end{table}

\subsubsection{Classification of diffusion tensor images} 
Diffusion tensor imaging (DTI) is a technique used to study white matter tractography in the brain. It consists of magnetic resonance images obtained by the diffusion of water molecules in biological tissues. We consider the DTI data set provided in \citet{bhattacharya2017}, which contains data for $46$ subjects with $28$ HIV+ subjects and $18$ healthy controls. For each subject, the data is a collection of $75$ positive definite matrices of size $3\times 3$, obtained by measuring at $75$ sites along a fixed fiber tract of the splenium of the corpus callosum. The goal is to classify HIV infection for each site and determine which sites are most susceptible. We compare intrinsic GP with a few linear and non-linear baselines in terms of binary classification performance for individual sites. A natural underlying manifold for such a classifier is the $6$-dimensional manifold $\mathbb{S}^3_{++}$ consisting of all $3\times 3$ positive definite matrices. However, a further investigation of the data set implies that those $3\times 3$ positive definite matrices correspond to points in $\mathbb{S}^2 \times \mathbb{S}^1 \times \mathbb{R}_{+,>}^3$, which is the underlying manifold we use to build our binary classifiers. More details are show in Appendix I.

We use 35 subjects (15 healthy, 20 HIV+) for training and 11 subjects for testing. We created 10 different batches of training and testing data by randomly sampling with replacement from the 46 subjects. 
The classification results are summarised as the cross entropy in Table~\ref{tab:DTIentropy}.The AUC-ROC scores are also presented in  Table~\ref{tab:DTIroc} in Appendix I. The top 5 most sensitive sites indexed by its location on the brain are identified by the SiGP classifier. The results of SiGP classifier are compared with the extrinsic GP (RBF kernel with $\R^7$ distance), SVM classifier and logistic regression. For a naive classifier with $P(y_i=0)=P(y_i=1)=0.5$, the corresponding cross entropy is 0.69. It is clear that the SiGP classifier significantly outperforms all other methods in both measures. Moreover, classification results for all the $75$ sites are summarised in Table~\ref{tab:DTIallarc} in Appendix \ref{appendix:dti}.

\begin{table}[t]
\caption{Cross entropy of the diffusion tensor imaging data classification. Values in parentheses show the standard deviation. 
}  
\label{tab:DTIentropy}
\vskip 0.15in
\begin{center}
\begin{small}
\begin{sc}
\begin{tabular}{lccccccccr}
\hline
    & SiGP & exGP &SVM  &LR \\
\hline
site13  & 0.48(0.08)& 0.68(0.02) & 0.63(0.11)&0.65(0.03) \\
site55 &  0.51(0.06) & 0.72(0.05) & 0.59(0.06)& 0.64(0.03) \\ 
site29 &  0.52(0.06) &   0.69(0.01) & 0.64(0.11)& 0.66(0.03) \\
site39 & 0.52(0.07) & 0.69(0.01)& 0.64(0.16)& 0.65(0.03) \\
site50 &  0.54(0.04) & 0.69(0.02)& 0.62(0.10)&0.65(0.03) \\
\hline
\end{tabular}
\end{sc}
\end{small}
\end{center}
\vskip -0.1in
\end{table}



\section{Conclusion}
In this work, we propose a novel approach of constructing the strip intrinsic Gaussian process on manifolds such as orthogonal groups, unitary groups, Stiefel manifolds and Grassmannian manifolds. The heat kernel of a manifold is used as the covariance function of the intrinsic Gaussian process, which can be estimated as the transition density of the Brownian motion on the manifold. The ball algorithm and the strip algorithm are developed to estimate the transition density of the Brownian motion. While the ball algorithm is applicable to a more general class of manifolds, the strip algorithm is proven to be more efficient, both mathematically and experimentally. We also compare the performance of the extrinsic method proposed in \cite{extrinsicGP} with that of our novel method on torus knots and high dimensional projective spaces.
The comparison in section \ref{sec:examples} indicates that the strip intrinsic Gaussian process achieves a significant improvement over the extrinsic Gaussian process. On the one hand, although there is an abundant interest in optimization on manifolds \cite{absil2009optimization,boumal2014manopt,ring2012optimization,vandereycken2013low}, we notice that most algorithms are based on the gradient descend method. On the other hand, the Bayesian optimization on Euclidean spaces is proven to be extremely effective in numerous scenarios. In future, we would like to combine these approaches with the intrinsic Gaussian process to perform the Bayesian optimisation on manifold.


\nocite{langley00}

\bibliography{example_paper}
\bibliographystyle{icml2024}

\newpage
\appendix
\onecolumn
\section{ Matrix Lie groups}
\label{Liegroup}
In this subsection, we will briefly review some basic facts about matrix Lie groups which we will use in this paper. Interested readers are refereed to standard resources on differential geometry such as \cite{helgason2001differential},\cite{lee2012introduction} and \cite{warner2013foundations}.

Let $\mathbb{F}$ be either real or complex number field. We denote by $\mathbb{F}^{n\times n}$ the space of all $n\times n$ matrices over $\mathbb{F}$ and we denote by $\GL(n,\mathbb{F}) \subseteq \mathbb{F}^{n\times n}$ the group consisting of all invertible $n\times n$ matrices. Let $G$ be a Lie subgroup of $\GL(n,\mathbb{F})$ and let $\mathfrak{g}$ be its Lie algebra. For each $A\in G$, the \emph{tangent space} of $G$ at $A$ is 
\begin{equation}\label{eqn:tagent space Lie group}
T_A G = A \mathfrak{g}.
\end{equation}
Hence we may write a tangent vector of $G$ at $A$ as $AX$, where $X\in \mathfrak{g}$. 

There is a canonical Riemannian metric $g^c$ on $G$. To be more precise, we have a positive definite bilinear form $g^c_A$ on $T_AG$ for each $A\in G$ defined by
\begin{equation}\label{eqn:metric Lie group}
g^c_A(AX,AY) = \tr(X^\ast Y),\quad X,Y\in \mathfrak{g},
\end{equation}
and $g^c_A$ varies smoothly with respect to $A$. For a given $X\in \mathfrak{g}$, we denote by $\lVert X \rVert$ the \emph{norm} of $X$ with respect to $g^c$, i.e., 
\begin{equation}
\lVert X \rVert := \sqrt{\tr(X^{\ast} X)}.
\end{equation}
It is clear that the metric $g^c$ on $G$ is bi-invariant if $G = \O(n),\SO(n),\U(n)$ or $\SU(n)$ where
\begin{enumerate}
\item general linear group: $\GL(n,\mathbb{R}):= \left\lbrace
A\in \mathbb{R}^{n\times n}:  \det (A) \ne 0  \right\rbrace$;
\item special linear group: $\SL(n,\mathbb{R}) := \left\lbrace
A\in \GL(n,\mathbb{R}): \det (A) = 1  \right\rbrace$;
\item orthogonal group: $\O(n) := \left\lbrace
A\in \GL(n,\mathbb{R}): A^\tp A = \I_n
 \right\rbrace$;
\item special orthogonal group: $\SO(n) = \O(n) \cap \SL(n,\mathbb{R})$;
\item unitary group: $\U(n) := \left\lbrace
A\in \GL(n,\mathbb{C}): A^\ast A = \I_n
 \right\rbrace$;
\item special unitary group: $\SU(n) = \U(n) \cap \SL(n,\mathbb{C})$.
\end{enumerate}
For instance, $\SO(2)$ consists of all $2\times 2$ matrices of the form 
\[
A = \begin{bmatrix}
\cos \theta & \sin \theta \\
-\sin \theta & \cos \theta
\end{bmatrix}, \quad \theta \in [0,2 \pi ).
\]
Correspondingly, the Lie algebra $\mathfrak{o}(2)$, which is defined to be the tangent space of $\SO(2)$ at the identity $\I_2\in \SO(2)$, consists of all $2\times 2$ skew symmetric matrices, i.e., 
\[
X = \begin{bmatrix}
0 & -a \\
a & 0
\end{bmatrix},\quad a\in \mathbb{R}.
\]

We summarize some important properties of these matrix Lie groups in Table \ref{tab:properties of Lie groups}.

\begin{table}[h!]
\def~{\hphantom{0}}
\begin{center}
\caption{topological properties of matrix Lie groups}{%
\begin{tabular}{lcccc}
 \\
                           & connectedness &  compactness & Lie algebra & dimension   \\ \hline
$\O(n)$ & no & yes & $\mathfrak{so}(n)$ & $n(n-1)/2$ \\\hline
$\SO(n)$ & yes & yes & $\mathfrak{so}(n)$ & $n(n-1)/2$  \\\hline
$\U(n)$ & yes & yes & $\mathfrak{u}(n)$ & $n^2$ \\\hline
$\SU(n)$ & yes & yes & $\mathfrak{su}(n)$ & $n^2 - 1$ \\\hline
\end{tabular}}
\label{tab:properties of Lie groups}
\end{center}
\end{table}

 Here 
$\mathfrak{so}(n)$ is the space of all skew-symmetric $n\times n$ real matrices; $\mathfrak{u}(n)$ is the space of all skew-Hermitian $n\times n$ complex matrices and $\mathfrak{su}(n)$ consists of all traceless skew-Hermitian $n\times n$ complex matrices.
 From now on, we assume that $G = \O(n),\SO(n),\U(n)$ or $\SU(n)$. With the canonical metric $g^c$, a \emph{geodesic curve} $\gamma(t)$ passing through $A\in G$ with the tangent direction $AX\in T_A \mathfrak{g}$ is given by
\begin{equation}\label{eqn:geodesic Lie group}
\gamma(t):= A \exp(t X) =A \sum_{j=0}^\infty \frac{(tX)^k}{k!}.
\end{equation}
The length of $\gamma(t) = A \exp(t X)$ is calculate by 
\begin{equation}
L(\gamma):= \int_0^t \sqrt{g_{\gamma(s)}(\dot{\gamma}(s),\dot{\gamma}(s))} ds =t  \lVert X \rVert .
\end{equation}

\begin{lemma}\label{lemma:geodesic distance Lie group}
Let $A,B$ be two points on $G = \SO(n)$ or $\SU(n)$ and let $\gamma(t)$ be the geodesic connecting $A$ and $B$. The geodesic distance between $A$ and $B$ is 
\begin{equation}\label{eqn:geodesic distance on Lie group}
d(A,B)  =\sqrt{\sum_{j=1}^n \lvert \log (\lambda_j)} \rvert^2, 
\end{equation}
where $\lambda_1,\dots, \lambda_n$ are eigenvalues of $A^\ast B$ and $\log(\lambda)$ is the principal logarithm of $\lambda\in \mathbb{C}$.
\end{lemma}
\begin{proof}
We only prove for $\SO(n)$ as the proof for $\SU(n)$ is similar. The Lie algebra of $\SO(n)$ is 
\begin{equation}\label{eqn:lie algebra o(n)}
\mathfrak{so}(n):= \lbrace
X\in \mathbb{R}^{n\times n}: X^\tp + X = 0
\rbrace.
\end{equation} 
According to \eqref{eqn:geodesic Lie group}, the geodesic passing through $A\in \SO(n)$ with the tangent direction $AX\in T_A\SO(n)$ is $\gamma(t)  = A \exp(tX)$. If $\gamma(1) = B$, then $A^\ast B = \exp(X)$. Since $X$ is a skew-symmetric matrix, the spectral theorem implies that 
\begin{equation}\label{eqn:svd skew-symmetric}
X = D^{\ast} \Sigma D,
\end{equation}
where $D\in \U(n)$ and $\Sigma$ is a diagonal matrix whose diagonal entries are eigenvalues of $X$. Explicitly, we have 
\[
\Sigma = \begin{bmatrix}
i\mu_1 &0 &  \cdots &  0 & 0  & 0 &\cdots &0 \\
0 & -i\mu_1  & \cdots  & 0 & 0 & 0 & \cdots &0\\
\vdots & \vdots & \ddots & \vdots & \vdots & \vdots & \ddots  &\vdots \\
0 & 0 & \cdots & i\mu_r  & 0  & 0 & \cdots  & 0  \\
0 & 0 & \cdots & 0 & -i\mu_r & 0 &  \cdots & 0 \\
0 & 0 & \cdots & 0 & 0 & 0&  \cdots & 0\\
\vdots & \vdots & \ddots & \vdots & \vdots & \vdots&  \ddots & \vdots\\
0 & 0 & \cdots & 0 & 0 & 0&  \cdots & 0\\
  \end{bmatrix},
\]
where $\mu_1,\dots,\mu_r$ are positive real numbers. Hence we may write the geodesic $\gamma(t)$ as
\begin{equation}\label{eqn:geodesic O(n)}
\gamma(t)  = A U^{\ast} \exp(t\Sigma) U,
\end{equation}
Moreover, we notice that $X$ is in fact a real matrix, the norm of $X = D^{\ast} \Sigma D$ is 
\begin{equation}\label{eqn:norm O(n)}
\lVert X \rVert_{\mathbb{R}} = \lVert \Sigma \lVert_{\mathbb{C}} = \sqrt{2\sum_{j=1}^r \mu_j^2} = \sqrt{\sum_{j=1}^n \lvert \log(\lambda_j) \rvert^2},
\end{equation}
since $\lVert X \rVert_{\mathbb{R}}^2 = \tr(X^\tp X) = \tr(X^\ast X) = \lVert \Sigma \rVert_{\mathbb{C}}^2$.
\end{proof}

\begin{remark}
Let $\mathbb{F}$ be real or complex number field and let $X\in \mathbb{F}^{n\times n}$ be such that $X^\ast  + X = 0$. Suppose $X = U^\ast \Sigma U$ is an eigendecomposition of $X$, where $U\in \U(n)$ and $\Sigma$ is a diagonal matrix whose diagonal entries are eigenvalues $\lambda_1,\dots, \lambda_n$ of $X$. Then the exponential $\exp(X)$ can be efficiently computed by 
\begin{equation}\label{eqn:exponential on U(n)}
\exp(X) = U^\ast D U
\end{equation}
where $D$ is the diagonal matrix whose diagonal entries are $e^{\lambda_1},\dots, e^{\lambda_n}$. 
If moreover we have $\tr(X) = 0$, then nonzero eigenvalues of $X$ must be of the form $\pm i \mu_1,\dots, \pi i \mu_r$ where $2r \le n$. Hence $\exp(X)$ is calculated by  
\begin{equation}\label{eqn:exponential on so}
\exp(X) = U^\ast \begin{bmatrix}
e^{i\mu_1} & 0 & \cdots & 0 & 0 & 0 & \cdots & 0 \\
0                          & e^{-i\mu_1} & \cdots  & 0 & 0 & 0 & \cdots & 0 \\
\vdots & \vdots & \ddots & \vdots & \vdots & \vdots &  \ddots & \vdots \\
0 & 0 & \cdots & e^{i\mu_r} & 0 & 0 & \cdots & 0 \\
0 & 0 & \cdots & 0 & e^{-i\mu_r} & 0 & \cdots & 0 \\
0 & 0 & \cdots & 0 & 0 &  1 & \cdots & 0 \\
\vdots & \vdots & \ddots  & \vdots & \vdots & \vdots & \ddots & \vdots \\
0 & 0 & \cdots & 0 & 0 &  0 & \cdots & 1 \\
\end{bmatrix} U.
\end{equation}
\end{remark}

\section{ Homogeneous spaces}\label{appendix:homogeneous spaces}
A \emph{homogeneous space} is a quotient space $G/H$ where $G$ is a Lie group and $H$ is a closed Lie subgroup of $G$. For each $A\in G$, we denote by $[A]$ the coset $AH$, which is an element in $G/H$. For simplicity, we assume that $G\subseteq \mathbb{R}^{n\times n}$ is a compact, semi-simple matrix Lie group. By definition, $G$ acts transitively on $G/H$ via the action:
\[
G\times G/H \to G/H,\quad (B,[A]) \to [BA].
\] 
Hence for each $[A] \in G/H$, the \emph{tangent space} of $G/H$ at $[A]$ is:
\[
T_{[A]} G/H = A \mathfrak{g}/\mathfrak{h} = \lbrace A(X + \mathfrak{h}): X\in \mathfrak{g} \rbrace.
\]

Let $g^c$ be the Riemannian metric defined in \eqref{eqn:metric Lie group} on $G$. There exists a linear subspace $\mathfrak{m} \subseteq \mathfrak{g}$ such that 
\begin{equation}\label{eqn:homogeneous space tangent space}
T_{[\I_n]} G/H \simeq \mathfrak{m},\quad \mathfrak{g} = \mathfrak{m} \oplus \mathfrak{h}.
\end{equation}
The \emph{geodesic} $\gamma$ passing through $[A]\in G/H$ with the tangent direction $X\in \mathfrak{m}$ is given by 
\begin{equation}\label{eqn:homogeneous space geodesic}
\gamma(t) = [A\exp(tX)].
\end{equation}

\subsection{Stiefel manifold}\label{subsubsec:Stiefel}
Let $\mathbb{F}$ be real or complex number field. For each pair of positive integers $k < n$, we define the \emph{Stiefel manifold} $\V(k,n)$ of orthonormal $k$-frames in $\mathbb{F}^n$ to be 
\begin{equation}\label{eqn:Stiefel definition}
\V_{\mathbb{F}}(k,n):= \lbrace 
A\in \mathbb{F}^{n\times k}: A^{\ast} A = \I_k
\rbrace.
\end{equation}
The \emph{tangent space} of $\V_{\mathbb{F}}(k,n)$ at $A\in \V_{\mathbb{F}}(k,n)$ is 
\[
T_{A} \V_{\mathbb{F}}(k,n) = \lbrace 
\Delta\in \mathbb{F}^{n\times k}: A^{\ast} \Delta + \Delta^{\ast} A = 0
\rbrace.
\]

We notice that $\V_{\mathbb{R}}(k,n)$ (resp. $\V_{\mathbb{C}}(k,n)$) is a homogeneous space, i.e., it is the quotient of $\SO(n)$ (resp. $\SU(n)$) by its subgroup $\SO(n-k)$ (resp. $\SU(n-k)$):
\begin{align*}
\V_{\mathbb{R}}(k,n) &\simeq \O(n)/\O(n-k) \simeq  \SO(n) / \SO(n-k), \\
\V_{\mathbb{C}}(k,n) &\simeq \U(n)/\U(n-k) \simeq  \SU(n) / \SU(n-k). 
\end{align*}

The quotient map $q:\SO(n)\to \SO(n)/\SO(n-k)\simeq \V_{\mathbb{R}}(n,k)$ is simply given by $A \to A \I_{n,k}$ where 
\[
\I_{n,k} = \begin{bmatrix}
\I_{k} \\
0
\end{bmatrix} \in \mathbb{R}^{n\times k}.
\] 
The quotient map $\SU(n)\to \SU(n)/\SU(n-k)\simeq \V_{\mathbb{C}}(n,k)$ is defined in the same way.

In particular, if $k= 1$, then $\V_{\mathbb{R}}(k,n)\simeq \mathbb{S}^{n-1}$ and the map $q$ under this identification becomes
\begin{equation}\label{eqn:fibration sphere}
q: \SO(n) \to \mathbb{S}^{n-1},\quad \pi(A) = Ae_1,
\end{equation} 
where $e_1 = (1,0,\dots, 0)^\tp$. The fiber of $q$ is obviously homeomorphic to $\SO(n-1)$ and hence $q$ is a fibration of $\SO(n)$ over $\mathbb{S}^{n-1}$ whose typical fiber is $\SO(n-1)$ Since $A\in \SO(2)$ can be parametrized as 
\[
A = \begin{bmatrix}
\cos \theta & \sin\theta \\
- \sin\theta & \cos \theta
\end{bmatrix}
\]
Hence it is straightforward to verify that $q:\SO(2) \to \mathbb{S}^1$ is a homeomorphism. 

By \citep[Section 2.4.1]{edelman1998geometry}, the \emph{geodesic} curve passing through $A\in \V_{\mathbb{F}}(k,n)$ with the tangent direction $X \in T_A \V_{\mathbb{F}}(k,n)$ is
\begin{equation}\label{eqn:Stiefel geodesic}
\gamma(t) = A M(t) + Q N(t),
\end{equation}
where $QR = (\I_{n} - A A^\tp)X$ is the QR-decomposition of $(\I_{n} - A A^\tp)X$ and $M(t)$ and $N(t)$ are $k\times k$ matrices determined by
\[
\begin{bmatrix}
M(t) \\
N(t)
\end{bmatrix} = \exp(t \begin{bmatrix}
A^\tp X &  -R^\tp\\
R & 0 
\end{bmatrix}) \I_{2k,k}.
\]
Here the Riemannian metric $g^c$ equipped on $\V_{\mathbb{F}}(k,n)$ is the \emph{canonical metric} induced from the quotient space structure of $\V_{\mathbb{F}}(k,n)$. To be more precise, 
\[
g_A^c(\Delta,\Delta) := 2\tr \Delta^\tp (\I_n - AA^\tp) \Delta.
\]
With this canonical metric, it is straightforward to verify that a geodesic starting from $A\in \V_{\mathbb{F}}(k,n)$ with the tangent direction $PX\I_{n,k}\in T_A \V_{\mathbb{F}}(k,n)$ can also be written as 
\begin{equation}\label{eqn:Stiefel godeosic0}
\gamma(t) = P \exp(tX) \I_{n,k},
\end{equation}
where $P$ is an orthogonal matrix (resp. unitary matrix) such that $ P \I_{n,k} = A$. 

\begin{proposition}\label{prop:length geodesic Stiefel}
Let $\gamma$ be the geodesic curve starting from $A\in \V_{\mathbb{R}}(k,n)$ with the direction $PX\I_{n,k}\in T_{A}\V_{\mathbb{R}}(k,n)$ where $P,X$ are as above. The length of $\gamma$ is 
\begin{equation}
L(\gamma) = \sqrt{\int_{0}^\tau g^c_{\gamma(t)} (\gamma'(t),\gamma'(t)) dt} = \sqrt{ -\tr(Y^2) + 2\tr(ZZ^\tp)} = \lVert X \rVert_F.
\end{equation}
In particular, if we specialize to the case $k=1$, then the geodesic curve connecting $A,B\in \V(1,n) = \mathbb{S}^{n-1}$ is just the arc of the big circle from $A$ to $B$ and the geodesic distance between $A$and $B$ is simply
\begin{equation}\label{eqn:distance sphere}
d^c(A,B) = \theta,
\end{equation}
where $\theta\in [0,\pi]$ is the angle between $A$ and $B$, i.e., $\cos \theta = A^\tp B$.
\end{proposition}

\begin{proof}
The first part follows directly from the definition of $L(\gamma)$ and $g^c$. Hence it is only left to prove the statement for $k=1$. In this case, we notice that $A,B$ are simply column vectors of dimension $n$. We let $\theta \in [0,\pi]$ be the angle between $A$ and $B$. By definition of $X$, we must have $A^\tp X = 0$ and the system equation (11)--(13) in the main paper reduces to 
\begin{align*}
B &= mA  + nQ, \\
r Q & = X, \\
\begin{bmatrix}
m \\
n
\end{bmatrix} & = \exp(\begin{bmatrix}
0 & -r\\
r & 0
\end{bmatrix}) \I_{2,1},
\end{align*}
where $m,n,r$ are real numbers and $A,B,Q$ are unit column vectors of dimension $n$. Now it is straightforward to verify that $r = \theta$, $m = \cos \theta,n=\sin \theta$ and $\lVert X \rVert_F = \theta$.
\end{proof}

\begin{lemma}\label{lemma:Stiefel tangent space}
For any $Q\in \SO(n)$ such that $Q\I_{n,k} = A$, we have 
\[
T_{A} \V_{\mathbb{R}}(k,n)  = Q T_{\I_{n,k}} V_{\mathbb{R}}(k,n).
\]
\end{lemma}
\begin{proof}
We notice that 
\[
\dim T_{A} \V_{\mathbb{R}}(k,n) = \dim T_{\I_{n,k}} V_{\mathbb{R}}(k,n) = \dim Q T_{\I_{n,k}} V_{\mathbb{R}}(k,n).
\]
Hence it is sufficient to prove that $Q T_{\I_{n,k}} V_{\mathbb{R}}(k,n) \subseteq T_{A} \V_{\mathbb{R}}(k,n)$. To see this, we let $\Delta = \begin{bmatrix}
\Delta_1 \\
\Delta_2
\end{bmatrix}$ be an element in $T_{\I_{n,k}} V_{\mathbb{R}}(k,n)$ and write $Q = \begin{bmatrix}
A & A'
\end{bmatrix}$. Then we must have 
\[
A^\tp Q \Delta =A^\tp (A\Delta_1 + A'\Delta_2) = \Delta_1\in \mathfrak{o}(k),
\]
since $A^\tp A = \I_k$ and $Q\in \SO(n)$. Therefore, by definition of $T_{A} \V_{\mathbb{R}}(k,n)$, we see that $Q\Delta\in T_{A} \V_{\mathbb{R}}(k,n)$ and this completes the proof.
\end{proof}

\subsection{Grassmannian manifolds}\label{subsubsec:Grassmannian}
Let $k,n$ be two non-negative integers such that $k\le n$. The \emph{Grassmannian manifold $\Gr_{\mathbb{F}}(k,n)$} consisting of all $k$-dimensional linear subspaces of $\mathbb{F}^n$. Moreover, $\Gr_{\mathbb{F}}(k,n)$ is a homogeneous space:
\begin{align}
\label{eqn:Grass def1}
\Gr_{\mathbb{R}}(k,n) &\simeq \O(n)/(\O(k,n) \times \O(n-k,n)) \simeq \V(k,n)/\O(k), \\
\label{eqn:Grass def2}
\Gr_{\mathbb{C}}(k,n) &\simeq \U(n)/(\U(k,n) \times \U(n-k,n)) \simeq \V_{\mathbb{C}}(k,n)/\U(k), 
\end{align}
Here $\O(k,n) \times \O(n-k,n)$ is embedded in $\O(n)$ as a Lie subgroup via the embedding
\[
j: \O(k,n) \times \O(n-k,n) \to \O(n),\quad j(A_1,A_2) = \begin{bmatrix}
A_1  & 0 \\
0 & A_2
\end{bmatrix}.
\]
Similarly, $\U(k) \times \U(n-k)$ is regarded as a Lie subgroup of $\U(n)$ via the same map. 

We remakr that the diffeomorphism $\Gr_{\mathbb{R}}(k,n) \simeq \V_{\mathbb{R}}(k,n)/\O(k)$ (resp. $\Gr_{\mathbb{C}}(k,n) \simeq \V_{\mathbb{C}}(k,n)/\U(k)$) is based on the fact that every $k$-dimensional subspace of $\mathbb{F}^n$ admits an orthonormal basis, which is unique up to a rotation. To be more precise, let $\mathbb{A}$ be subspace in $\mathbb{F}^n$ of dimension $k$ and let $v_1,\dots, v_k$ be an orthonormal basis. Then we have $V = [v_1,\dots, v_k]\in \V_{\mathbb{F}}(k,n)$ and the column vectors of $VQ$ also form an orthonormal basis of $\mathbb{A}$, for any $Q\in \O(k)$ (resp. $Q\in \U(k)$). Therefore, in the rest of this paper, we simply represent an element $\mathbb{A}\in \Gr_{\mathbb{F}}(k,n)$ by an $n\times k$ matrix $Y_{\mathbb{A}}\in \V_{\mathbb{F}}(k,n)$, whose column vectors form an orthonormal basis of $\mathbb{A}$. The \emph{tangent space} of $\Gr_{\mathbb{F}}(k,n)$ at $\mathbb{A}$ is:
\begin{equation}\label{eqn:Grass tangent space}
T_{\mathbb{A}} \Gr_{\mathbb{F}}(k,n) = \lbrace 
\Delta \in \mathbb{F}^{n\times k}: Y_{\mathbb{A}}^{\ast} \Delta = 0
\rbrace
\end{equation}
The canonical metric on Stiefel manifolds we discussed in Section~\ref{subsubsec:Stiefel} induces a Riemannian metric on $\Gr(k,n)$ via its quotient space structure \eqref{eqn:Grass def1} and \eqref{eqn:Grass def2}. Equipping with this induced metric, the \emph{geodesic curve} emanating from $\mathbb{A}\in \Gr_{\mathbb{F}}(k,n)$ with tangent direction $\Delta \in T_{\mathbb{A}} \Gr_{\mathbb{F}}(k,n)$ is given by 
\begin{equation}\label{eqn:Grass geodesic}
\mathbb{A}(t) = \operatorname{span} \left( \begin{bmatrix}
Y_{\mathbb{A}} V & U 
\end{bmatrix} \begin{bmatrix}
\cos \Sigma t \\
\sin \Sigma t 
\end{bmatrix} V^\tp \right),
\end{equation}  
where $U \Sigma V^\tp$ is the compact singular value decomposition of $\Delta$ and $\operatorname{span}(Y)$ means the linear subspace spanned by column vectors of a matrix $Y$. We refer interested readers to \citep[Section 2.5.1]{edelman1998geometry} for detailed computation of geodesic curves on $\Gr_{\mathbb{F}}(k,n)$. Moreover, the geodesic distance between $\mathbb{A},\mathbb{B}\in \Gr(k,n)$ is simply 
\begin{equation}\label{eqn:distance Grassmann}
d^c (\mathbb{A},\mathbb{B}) = \sqrt{\sum_{j=1}^k \theta_j^2},
\end{equation}
where $ \cos(\theta_j)$ is the $j$-th singular value of $A^\tp B$ and $A$ (resp. $B$) is the $n\times k$ matrix such that $A^\tp A = \I_k$ (resp. $B^\tp B = \I_k$) representing $\mathbb{A}$ (resp. $\mathbb{B}$).

In particular, $\Gr_{\mathbb{R}}(1,n) \simeq \mathbb{P}_{\mathbb{F}}^{n-1}$, the $(n-1)$-dimensional projective space over $\mathbb{F}$, consisting of all lines pass through the origin in $\mathbb{F}^n$. We remind the audiences who are not familiar with projective spaces that 
\begin{equation}\label{eqn:real projective space}
\mathbb{P}_{\mathbb{R}}^{n-1} \simeq \mathbb{S}^{n-1}/\mathbb{Z}_2,
\end{equation}
obtained from $\mathbb{S}^{n-1}$ by identifying the antipodal points. Over $\mathbb{C}$, we have 
\begin{equation}\label{eqn:complex projective space}
\mathbb{P}_{\mathbb{C}}^{n-1} \simeq \mathbb{S}^{2n-1}/\mathbb{S}^1,
\end{equation} 
identifying $x\in \mathbb{S}^{2n-1}\subseteq \mathbb{C}^{n}$ with points in $ \mathbb{S}^{2n-1}$ of the form $\lambda x$ for $\lambda \in \mathbb{S}^1 \subseteq \mathbb{C}$. In this case, formula \eqref{eqn:distance Grassmann} becomes 
\begin{equation}\label{eqn:distance projective space}
d^c(\mathbb{A},\mathbb{B}) = \lvert \theta \rvert,
\end{equation}
where $ \theta = \arccos(A^\ast B)\in [-\pi/2,\pi/2)$ and $A$ (resp $B$) is an element in $\V_{\mathbb{F}}(1,n)$ representing $\mathbb{A}$ (resp. $\mathbb{B}$).
\section{Brownian motion on matrix manifolds manifolds}\label{appendix:GP on matrix manifold}
In this section, we provide details of how to derive algorithms to simulate Brownian paths on Lie groups, Stiefel manifolds and Grassmannian manifolds (cf. Algorithms~\ref{alg:simulation of BM on O(n)}--\ref{alg:Brownian paths on projective spaces}).
\subsection{Brownian motion on $\O(n)$, $\SO(n)$, $\U(n)$ and $\SU(n)$}\label{subsec:GP on SO(n)} 
We notice by Table~\ref{tab:properties of Lie groups} that $\O(n)$ is a disconnected manifold, whose identity component is $\SO(n)$. Therefore, it is sufficient to simulate the Brownian motion on $\SO(n)$. By Equation~\eqref{eqn:geodesic O(n)}, geodesic curves on $\SO(n)$ are explicitly known, Algorithm \ref{alg:simulation of BM} can be applied directly. We recall that by Equation~\eqref{eqn:lie algebra o(n)} the tangent space of $\SO(n)$ at $A\in \SO(n)$ is simply a left translation of $\mathfrak{o}(n)$ by $A$, where $\mathfrak{o}(n)$ is the space of $n\times n$ skew symmetric matrices. Moreover, since $\mathfrak{o}(n)$ is a vector space, a Brownian path $W(t)$ in $\mathfrak{o}(n)$ is of the form $W(t) = (W_{ij}(t))$ where $W_{ii}(t) = 0,i=1,\dots, n$, $W_{ij}(t)$ is a Brownian path on $\mathbb{R}$ and $W_{ji}(t) = -W_{ij}(t),1\le i < j \le n$. 
On $\U(n)$ (resp. $\SU(n)$), the algorithm for the simulation of the Brownian motion sample paths is similar. We simply replace the Brownian motion sample path $W(t)\in \mathfrak{o}(n)$ in Algorithm~\ref{alg:simulation of BM on O(n)} by a Brownian motion sample path in $\mathfrak{u}(n)$ (resp. $\mathfrak{su}(n)$).
\subsection{Brownian motion on Stiefel manifold}\label{subsec:GP on Stiefel} For notational simplicity, we only discuss the real Stiefel manifolds, as the case of complex Stiefel manifolds is exactly the same. We recall that given positive integers $k \le n$,  
the Stiefel manifold over $\mathbb{R}$ is defined as 
$\V_{\mathbb{R}}(k,n) := \lbrace
A \in \mathbb{R}^{n\times k}: A^\tp A = \I_k
\rbrace.$
\if In particular, if $k=1$ then $\V_{\mathbb{R}}(k,n)$ is simply the $(n-1)$-dimensional sphere $\mathbb{S}^{n-1}$. For instance, $\V_{\mathbb{R}}(1,2)  = \left\lbrace
(\cos \theta, \sin \theta)^\tp: \theta \in [0,2\pi)
\right\rbrace = \mathbb{S}^1$.
Another extreme example is $k = n$, in which case $\V_{\mathbb{R}}(k,n)$ is the orthogonal group $\O(n)$ discussed in Section~\ref{subsec:GP on SO(n)}.\fi  To estimate the heat kernel on $\V_{\mathbb{R}}(k,n)$, we need to \if compute the distance between $A,B\in \V_{\mathbb{R}}(k,n)$, which amounts to\fi find $X\I_{n,k}\in T_{A} \V_{\mathbb{R}}(k,n)$ such that the geodesic determined by $A$ and $X\I_{n,k}$ passes through $B$. According to equation \eqref{eqn:Stiefel geodesic}, $X$ can be calculated by solving the system:
\begin{align}
\label{eqn:Stiefel-tangent direction1}
B &= A M + QN, \\
\label{eqn:Stiefel-tangent direction2}
QR & = (\I_n - A A^\tp) X, \\
\label{eqn:Stiefel-tangent direction3}
\begin{bmatrix}
M \\
N
\end{bmatrix} & = \exp \left( \begin{bmatrix}
A^\tp X & -R^\tp\\
R & 0
\end{bmatrix} \right) \I_{2k,k},
\end{align}
where $Q$ is a $n\times k$ matrix such that $Q^\tp Q = \I_k$ and $R$ is a $k\times k$ upper triangular matrix. In general, the system \eqref{eqn:Stiefel-tangent direction1}-\eqref{eqn:Stiefel-tangent direction3} has no explicit solution, but one can solve it explicitly if $k=1$ by Proposition~\ref{prop:length geodesic Stiefel}. \if According to Algorithm \ref{alg:simulation of BM}, we need two ingredients to simulate Brownian paths on a Stiefel manifold $\V_{\mathbb{R}}(k,n)$: Brownian paths on the tangent space $T_{A} \V_{\mathbb{R}}(k,n)$ and the geodesic curve passing through $A$ with tangent direction $X\in T_{A} \V_{\mathbb{R}}(k,n)$.\fi \if To do concrete computations, we write an element $A\in \V_{\mathbb{R}}(k,n)$ as an $n\times k$ matrix such that $A^\tp A = \I_k$, rather than an equivalence class of an orthogonal matrix in the quotient $\SO(n)/\SO(n-k)$.\fi Although \eqref{eqn:Stiefel godeosic0} provides us an explicit expression for geodesics, it is not obvious how we could simulate Brownian paths on $T_{A} \V_{\mathbb{R}}(k,n) = \{\Delta \in \mathbb{R}^{n\times k}: A^\tp \Delta + \Delta^\tp A = 0\}$ for arbitrary $A\in \V_{\mathbb{R}}(k,n)$. To this end, we first observe that  
\begin{equation}\label{eqn:Stiefel tangent space identity}
T_{\I_{n,k}} \V_{\mathbb{R}}(k,n) = \left\lbrace 
\begin{bmatrix}
\Delta_1 \\
\Delta_2
\end{bmatrix}: \Delta_1\in \mathfrak{o}(k),\Delta_2\in \mathbb{R}^{(n-k) \times k}
\right\rbrace.
\end{equation} 
Lemma~\ref{lemma:Stiefel tangent space} implies that at any point $A\in \V_{\mathbb{R}}(k,n)$, we could first simulate Brownian paths in $T_{\I_{n,k}} \V_{\mathbb{R}}(k,n)$, which is easy by \eqref{eqn:Stiefel tangent space identity}, then multiply the Brownian path by some $Q\in \SO(n)$ such that $Q \I_{n,k} =A$ to obtain a Brownian path on $T_{A} \V_{\mathbb{R}}(k,n)$. Thus, We obtain Algorithm~\ref{alg:Brownian paths on Stiefel} for simulations of Brownian motions on $\V_{\mathbb{R}}(k,n)$. If $k=1$, Algorithm~\ref{alg:Brownian paths on Stiefel} can be simplified. We have 
\begin{equation}\label{eqn:Brownian paths on sphere1}
X_{i,\delta}(\delta) = 
\begin{bmatrix}
0 \\
\epsilon_2\\
\vdots \\
\epsilon_n
\end{bmatrix}, \quad R = \sqrt{\sum_{j=2}^n \epsilon^2_j}, \quad Q = R^{-1}Q_{i-1}X_{i,\delta}(\delta)
\end{equation}
and $M = \cos R, N = \sin R$, which implies 
\begin{equation}\label{eqn:Brownian paths on sphere2}
A_i =  A_{i-1}\cos R  + R^{-1} Q_{i-1}X_{i,\delta}(\delta) \sin R.
\end{equation}

Combining \eqref{eqn:Brownian paths on sphere2} with algorithm~\ref{alg:Brownian paths on Stiefel}, we obtain Algorithm~\ref{alg:Brownian paths on spheres}.

\subsection{Brownian motion on Grassmannian manifolds} \label{grass manifold} Let $\mathbb{F}$ be $\mathbb{R}$ or $\mathbb{C}$. Given positive integers $k \le n $, the Grassmannian manifold $\Gr_{\mathbb{F}}(k,n)$ is defined by 
\[
\Gr_{\mathbb{F}}(k,n) := \lbrace 
\mathbb{A}\subseteq \mathbb{F}^n: \text{$k$-dimensional linear subspace}
\rbrace.
\] 
\if
Here $X^\ast $ denotes the conjugate transpose of a matrix $X$.
\fi
\if In particular, $\mathbb{P}^{n-1}_{\mathbb{F}}:=\Gr_{\mathbb{F}}(1,n)$ is called the $(n-1)$-dimensional projective space over $\mathbb{F}$. According to the definition, $\mathbb{P}^{n-1}_{\mathbb{F}}$ consists of all $n\times n$ matrices $A$ which can be written as $A = u u^\ast$, where $u$ is a unit column vector of dimension $n$. For instance, if $n = 2$ and $\mathbb{F} = \mathbb{R}$, then we can write $u = (\cos \theta, \sin \theta)^\tp \in \mathbb{R}^2$ and we have 
\begin{equation}\label{eqn:parametrization real P1}
\mathbb{P}^1_{\mathbb{R}} = \left\lbrace
\begin{bmatrix}
\cos^2 \theta & \cos \theta \sin \theta \\
 \cos \theta \sin \theta & \sin^2 \theta
\end{bmatrix}: \theta \in [0,2 \pi )
\right\rbrace.
\end{equation}  
Another example is for $n = 2$ and $\mathbb{F} = \mathbb{C}$. In this case, we have $u =(\cos \theta e^{i \phi}, \sin\theta e^{i \psi})\in \mathbb{C}^2$ with $\theta \in [0,\pi/2], \phi,\psi \in [0,2\pi)$ and  
\begin{equation}
\begin{aligned}\label{eqn:parametrization complex P1}
\mathbb{P}^1_{\mathbb{C}} = \left\lbrace
\begin{bmatrix}
\cos^2 \theta & \cos\theta \sin \theta e^{i (\phi - \psi)} \\
\cos\theta \sin \theta e^{i (\psi - \phi )}& \sin^2 \theta
\end{bmatrix}: \\ \theta \in [0,\pi/2], \phi,\psi \in [0,2\pi) \right\rbrace.
\end{aligned}
\end{equation} 
In general, $\Gr_{\mathbb{F}}(k,n)$ does not admit a simple parametrisation, like \eqref{eqn:parametrization real P1} and \eqref{eqn:parametrization complex P1} but there are several convenient ways to describe a point in $\Gr_{\mathbb{F}}(k,n)$.
Let $\mathbb{A}$ be a point in $\Gr_{\mathbb{F}}(k,n)$.\fi We denote by $Y_{\mathbb{A}}$ an $n\times k$ matrix whose column vectors form an orthonormal basis of $\mathbb{A}$. According to Appendix~\ref{subsubsec:Grassmannian}, the geodesic starting from $\mathbb{A}$ with tangent direction $\Delta \in T_{\mathbb{A}} \Gr_{\mathbb{F}}(k,n) = \{\Delta\in \mathbb{F}^{n\times k}: Y_{\mathbb{A}}^\tp \Delta = 0\}$ is represented by the curve 
\begin{equation}\label{eqn:GP on Grass geodesic} 
Y(t) = \begin{bmatrix}
Y_{\mathbb{A}}V & U 
\end{bmatrix} \begin{bmatrix} 
\cos \Sigma t  \\
\sin \Sigma t
\end{bmatrix} V^{\ast}
\end{equation}
in $V_{\mathbb{F}}(k,n)$. Here $U \Sigma V^{\ast}$ is the compact singular value decomposition of $\Delta$. Moreover, we notice that if $Y^\perp_{\mathbb{A}}$ is any $n\times (n-k)$ matrix whose column vectors form an orthonormal basis of the complement of $\mathbb{A}$ in $\mathbb{F}^n$, then $\Delta \in T_\mathbb{A} \Gr_{\mathbb{F}}(k,n)$ can also be written as 
\begin{equation}
\Delta = Y^\perp_{\mathbb{A}} H,\quad H\in \mathbb{F}^{(n-k) \times k}.
\end{equation}
We write diagonal entries of $\Sigma$ as $\sigma_1,\dots, \sigma_k$ and we have Algorithm~\ref{alg:Brownian paths on Grassmann manifolds} to simulate a Brownian path on $\Gr(k,n)$ starting from $\mathbb{A}$. We remark that $\Gr_{\mathbb{F}}(1,n)$ is of particular interest since in this case we have $\Gr_{\mathbb{F}}(1,n) = \mathbb{P}_{\mathbb{F}}^{n-1}$. We recall from \eqref{eqn:real projective space} that $\mathbb{P}_{\mathbb{R}}^{n-1}$ is simply the quotient of $\mathbb{S}^{n-1}$ by a $\mathbb{Z}_2$-action. 
Hence the heat kernel on $\mathbb{P}_{\mathbb{R}}^{n-1}$ can be efficiently estimated by the combination of Algorithm~\ref{alg:Brownian paths on spheres} and \eqref{eqn:heat kernel quotient manifold}. The case of $\mathbb{P}_{\mathbb{C}}^{n-1}$ is more subtle. According to \eqref{eqn:complex projective space}, we have that $\mathbb{P}_{\mathbb{C}}^{n-1}$ is the quotient of $\mathbb{S}^{2n-1}$ by an action of $\mathbb{S}^1$. Hence we have an integral formula \cite{elworthy1982stochastic,ndumu1996integral} for the heat kernel on $\mathbb{P}_{\mathbb{C}}^{n-1}$. It turns out that this integral formula is not practically useful. To this end, we may specialize Algorithm~\ref{alg:Brownian paths on Grassmann manifolds} to obtain Algorithm~\ref{alg:Brownian paths on projective spaces} for the simulation of Brownian paths starting from a fixed $[v]$ on $\mathbb{P}_{\mathbb{C}}^{n-1}$, where $[v]$ denotes the line in $\mathbb{C}^n$ uniquely determined by the unit norm vector $v\in \mathbb{S}^{2n-1} \subseteq \mathbb{C}^n\setminus \{0\}$.

\section{Heat kernel}\label{appendix:heat kernel}
Let $M$ be a Riemannian manifold with metric $g^M$ and let $G$ be a group acting on $M$ freely, properly and isometrically. Then we have Proposition~\ref{prop:heat kernel quotient manifold}, for which we give a proof here due the lack of appropriate reference, although this fact is well-known to the community of differential geometry.
\begin{proposition}\label{prop:heat kernel quotient manifold}
There exists a unique metric on the quotient manifold $X = M/G$ such that 
\begin{equation}\label{eqn:heat kernel quotient manifold}
p^X(t,[x],[y]) =  \int_{G} p^M(t,x,gy),
\end{equation}
where $[x],[y]$ are points in $X$ represented by $x,y\in M$ respectively. In particular, if $G$ is a finite group, then 
\[
p^X (t,[x],[y]) =  \sum_{g\in G} p^M(t,x,gy).
\]
\end{proposition}
\begin{proof}
According to \citep[Proposition 2.20]{gallot1990riemannian}, there exists a unique metric on $X$, such that the quotient map $\pi: M \to M/G$ is a Riemannian covering map, i.e., $\pi$ is a locally isometric smooth covering map. One can obtain \eqref{eqn:heat kernel quotient manifold} by recalling the definition of the heat kernel.
\end{proof}

For ease of reference, we also record the next two simple formulae for heat kernels on the product manifold and open submanifold respectively, which can be easily verified by definition.
\begin{proposition}\label{prop:heat kernel product manifold}
If $M,M_1$ are Riemannian manifolds whose heat kernels are respectively $p^{M}$ and $p^{M_1}$, then the heat kernel on $M \times M_1$ is given by
\[
p^{M \times M_1} (t,(x,x_1),(y,y_1)) = p^{M}(t,x,y) p^{M_1}(t,x_1,y_1),\quad (x,x_1),(y,y_1)\in M \times M_1
\]
Moreover, for each open submanifold $U$ of $M$, the heat kernel on $U$ is
\[
p^{U} (t,x,y) = p^{M}(t,x,y),\quad (x,y)\in M.
\]
\end{proposition}

\section{Tubes on a manifold and Proof of theorem 3.1}\label{appendix: tubes}

We summarize some facts about the volume of a tube on a manifold in this subsection. Suppose that $M$ is a Riemannian manifold of dimension $m$. We define for each $x\in M$, $d_0 \ge 0,\epsilon >0$ the set
\[
S_x(d_0,\epsilon) := \left\lbrace
y\in M: \lvert d(x,y) -  d_0 \rvert < \epsilon 
\right\rbrace.
\]
We also denote by $S_x(d_0)$ the set consisting of all $y\in M$ such that $d(x,y) = d_0$.

\begin{lemma}\label{lemma:tube}
Let $M$ be a complete Riemannian manifold and let $D := \sup_{x,y\in M} d(x,y)$ be the diameter of $M$. We suppose that $M$ has the following property: for any $x,y\in M$ with $d:= d(x,y) < D$ and any distance minimizing unit speed geodesic curve $\gamma:[0,d]\to M$ connecting $x$ and $y$, there is some positive $\epsilon \le D - d$ such that $\overline{\gamma}$ is also distance minimizing. Here $\overline{\gamma}$ is the unit speed geodesic curve determined by
\[
\overline{\gamma}|_{[0,d]} = \gamma,\quad \overline{\gamma}'(d) = \gamma'(d).
\]
Then for sufficiently small $\epsilon > 0$, the strip $S_x(d_0,\epsilon)$ can be described as 
\[
S_x(d_0,\epsilon) = \left\lbrace
y \in M: d(y,S_{x}(d_0)) < \epsilon
\right\rbrace.
\]
That is, $S_x(d_0,\epsilon)$ is the tube of $S_{x}(d_0)$ of radius $\epsilon$.
\end{lemma}
\begin{proof}
We denote by $N_{\epsilon}$ the set consisting of $y\in M$ such that 
\[
d(y,S_x(d_0)) < \epsilon.
\] 
Given a point $y \in N_{\epsilon}$, we have by triangle inequality that 
\begin{align*}
 d(y,x)  &\le d(y,S_x(d_0)) + d(S_x(d_0), x) < \epsilon + d_0, \\
d_0  - \epsilon &\le d(S_x(d_0), x) - d(y,S_x(d_0))  \le d(y,x),
\end{align*}
which implies $N_{\epsilon} \subseteq S_x(d_0,\epsilon) $. For the other containment, we notice that if there exists some $y\in S_x(d_0,\epsilon) \setminus N_{\epsilon}$, then we must have 
\[
\lvert d(x,y) - d_0 \rvert < \epsilon \le d(y,S_x(d_0)).
\]
We let $\gamma: [0,d(x,y)] \to M$ be a unit speed distance minimizing curve connecting $x$ and $y$. By the assumption on $M$, we can find some $z\in S_x(d_0)$ such that $d(y,z) < \epsilon$, which contradicts the fact that 
\[
\epsilon \le d(y, S_x(d_0)) \le d(y,z).
\] 
Indeed, if $d_0 \le d(x,y)$, then we simply take $z = \gamma(d_0)$ and we take $z = \overline{\gamma}(d_0)$ otherwise.
\end{proof}

\begin{lemma}\label{lemma:submanifold}
The set $S_{x}(d_0,\epsilon)$ is an open submanifold of $M$. Moreover, for a generic $d_0$ the set $S_x(d_0)$ is a $(m-1)$-dimensional closed submanifold of $M$. To be more precise, the set of $d_0\in \mathbb{R}_+$ such that $S_x(d_0)$ is not smooth has measure zero.
\end{lemma}
\begin{proof}
We consider the map $\varphi:M \to \mathbb{R}_+$ defined by $\varphi(y)= d(x,y)$. It is obvious that $\varphi$ is a continuous function and we have $S_x(d_0,\epsilon) = \varphi^{-1}((d_0-\epsilon, d_0 + \epsilon))$. This implies that $S_x(d_0,\epsilon) $ is an open subset of $M$ and hence it is a submanifold. 

Since $\varphi$ is a smooth function from $M$ to $\mathbb{R}_+$, by Sard's theorem the set of critical values of $\varphi$ is of measure zero. This implies the set $S_{x}(d_0,\epsilon) = \varphi^{-1}(d_0)$ is a smooth closed submanifold of $M$ for all $d_0\in \mathbb{R}_+ \setminus E$ where $E$ is some measure zero subset. 
\end{proof}

Let $P$ be a $q$-dimensional submanifold of a Riemannian manifold $M$. We denote by $\Vol^M_P(\epsilon)$ the volume of the tube around $P$ in $M$ of radius $\epsilon > 0$. Then we have 
\begin{theorem}\citep[Theorem 9.23]{gray2012tubes} \label{thm:tube formula}
There is a power series expansion of $\Vol^M_P(\epsilon)$ in $\epsilon$: 
\[
\Vol^M_P(\epsilon) = \frac{(\pi \epsilon^2)^{\frac{1}{2} (m-q)}}{(\frac{1}{2} (m-q))!} \int_P \left( 1 + A \epsilon^2 + B \epsilon^4 + O(\epsilon^6) \right) dP,
\]
where $A,B$ are quantities determined by the curvature of $M$, curvature of $P$ and the second fundamental form of $P$ in $M$, which are independent to $\epsilon$.
\end{theorem} 

Let $\lambda$ be a constant number and let $K^m(\lambda)$ be a Riemannian manifold of dimension $m$ whose sectional curvature is identically $\lambda$. We denote by $V^{K^m(\lambda)}_x(r)$ the volume of the geodesic ball around $x\in K^m(\lambda)$ of radius $r$.
\begin{theorem}\citep[Corollary 3.18]{gray2012tubes}\label{thm:volume geodesic ball}
Let $r_0$ be the distance from $x$ to its cut locus on $K^m(\lambda)$. For any $0 < r \le  r_0$, we have
\begin{equation}
\Vol^{K^m(\lambda)}_x(r) = \frac{2\pi^{m/2}}{\Gamma(\frac{m}{2})} \int_0^r \left(  \frac{\sin(t\sqrt{\lambda})}{\sqrt{\lambda}} \right)^{m-1} dt.
\end{equation}
\end{theorem}

\begin{corollary}\label{cor:volume geodesic ball sphere}
For any $r \in (0,\pi ]$, we have 
\begin{equation}
\Vol^{\mathbb{S}^{m}}_x(r) = \frac{2\pi^{n/2}}{\Gamma(\frac{m}{2})} \int_0^r \left(  \sin(t)\right)^{m -1} dt.
\end{equation}
In particular, for any $x\in \mathbb{S}^m$ and $d_0\in [\epsilon, \pi-\epsilon]$ the volume of $S_x(d_0,\epsilon)$ is 
\begin{equation}\label{eqn:volume tube sphere}
\Vol(S_x(d_0,\epsilon)) = \Vol^{\mathbb{S}^{m}}_x(d_0 + \epsilon) - \Vol^{\mathbb{S}^{m}}_x(d_0 - \epsilon) = \frac{2\pi^{m/2}}{\Gamma(\frac{n}{2})} \int_{d_0 - \epsilon}^{d_0+\epsilon} \left( \sin(t) \right)^{m-1} dt.
\end{equation}
\end{corollary}

The proof of theorem 3.1 is:
\begin{proof}
The idea of the proof is to analyze the error of $\hat{p}^M_t(x,y)$. Since $\hat{p}^M_t(x,y)$ can be regarded as an average of the estimator in Algorithm~\ref{alg:estimation of heat kernel} on the strip $S_{x}(d_0,\epsilon)$, the error analysis of $\hat{p}^M_t(x,y)$ can be obtained by that of the estimator in Algorithm~\ref{alg:estimation of heat kernel}, which is given in \cite{niu2018intrinsic}.
\end{proof}

\newpage
\section{Summary of algorithms}\label{appendix:algorithms}
In this section, we record all algorithms discussed in Section~\ref{heatkernelmani}.

\begin{algorithm}[H]
 \caption{Simulation of Brownian paths on $\SO(n)$}
   \label{alg:simulation of BM on O(n)}
\begin{algorithmic}
   \STATE  Initialize $A_0 = A$.
   \FOR{ $i =1, 2, \dots$ }
   \STATE compute $X_{i-1}(\delta)\in \mathfrak{o}(n)$; (BM sample path in $\mathfrak{o}(n)$ with step variance $\delta$)
   \STATE compute svd: $X_{i-1}(\delta) = D^{\ast} \Sigma D$; (see equation \eqref{eqn:svd skew-symmetric})
   \STATE set $A_i = A_{i-1} D^\ast \exp(\Sigma) D $;
   \ENDFOR
\end{algorithmic}
\end{algorithm}

\begin{algorithm}[H]
 \caption{Simulation of Brownian paths on $\V_{\mathbb{R}}(k,n)$}
   \label{alg:Brownian paths on Stiefel}
\begin{algorithmic}
   \STATE  Initialize $A_0 = A$.
   \FOR{ $i =1, 2, \dots$ }
  \STATE find $Q_{i-1}\in \SO(n)$ such that $Q_{i-1}\I_{n,k} = A_{i-1}$;
   \STATE compute $X_{i-1} (\delta)$; (Brownian motion sample path in $T_{\I_{n,k}} \V_{\mathbb{R}}(k,n)$ with step variance $\delta$, see \eqref{eqn:Stiefel tangent space identity} )
   \STATE compute $M,N,Q$ for $A_{i-1}$ and $Q_{i-1}X_{i-1}(\delta)$; (see equation \eqref{eqn:Stiefel geodesic})
   \STATE set $A_i = A_{i-1} M + Q N $;
   \ENDFOR
\end{algorithmic}
\end{algorithm}

\begin{algorithm}[H]
 \caption{Simulation of Brownian paths on $\mathbb{S}^{n-1}$}
\label{alg:Brownian paths on spheres} 
\begin{algorithmic}
   \STATE  Initialize $A_0 = A$.
   \FOR{ $i =1, 2, \dots$ }
  \STATE find $Q_{i-1}\in \SO(n)$ such that $Q_{i-1}\I_{n,1} = A_{i-1}$;
   \STATE sample random values $\epsilon_2,\dots, \epsilon_n$ from $N(0,\sigma^2= \delta)$;
   \STATE set $X_{i-1}(\delta) = (0,\epsilon_2,\dots, \epsilon_n)^{\tp}$;
   \STATE compute $R$ and $Q$; $\{$see \eqref{eqn:Brownian paths on sphere1} $\}$
   \STATE compute $A_i$; $\{$see \eqref{eqn:Brownian paths on sphere2} $\}$
   \ENDFOR
\end{algorithmic}
\end{algorithm}

\begin{algorithm}[H]
  \caption{Simulation of Brownian paths on $\Gr_{\mathbb{F}}(k,n)$}
\label{alg:Brownian paths on Grassmann manifolds} 
\begin{algorithmic}
   \STATE  Initialize $Y_0 = Y_{{\mathbb{A}}}$.
   \FOR{ $i =1, 2, \dots$ }
  \STATE find an orthogonal complement $Y^\perp_{i-1}$ of $Y_{i-1}$;
   \STATE sample random values $\epsilon_{jl}$ from $N(0,\sigma^2= \delta)$ for $1\le j \le n-k,1\le l \le k$;
   \STATE set $\Delta_{i-1} = Y^\perp_{i-1} (\epsilon_{jl})$;
   \STATE compute compact svd: $\Delta_{i-1}  = U \Sigma V^{\ast}$;
    \STATE set $\Lambda = \operatorname{diag}(\cos (\sigma_1),\dots ,\cos(\sigma_k))$;\\
    \STATE  set $\Gamma = \operatorname{diag}(\sin (\sigma_1),\dots ,\sin(\sigma_k))$; \\
     \STATE set $Y_i = \begin{bmatrix} Y_{i-1}V & U \end{bmatrix} \begin{bmatrix} \Lambda \\
\Gamma
  \end{bmatrix} V^{\ast}$; \\
   \ENDFOR
\end{algorithmic}
\end{algorithm}

 \begin{algorithm}[tb]
  \caption{Simulation of Brownian paths on $\mathbb{P}_{\mathbb{C}}^{n-1}$}
\label{alg:Brownian paths on projective spaces} 
\begin{algorithmic}
   \STATE  Initialize $v_0 = v$.
   \FOR{ $i =1, 2, \dots$ }
  \STATE find an orthogonal complement $v^\perp_{i-1}$ of $v_{i-1}$;
   \STATE sample random values $\epsilon_{j}\in \mathbb{C}$ from $N(0,\sigma^2= \delta)$ for $1 \le j \le n-1$;
   \STATE compute $R = \sqrt{\sum_{j=1}^{n-1} \lvert \epsilon_j \rvert^2}$;
   \STATE set $v_i = \cos(R) v_{i-1} + \sin(R)  R^{-1}  v_{i-1}^{\perp} (\epsilon_j)$;
   \ENDFOR
\end{algorithmic}
\end{algorithm} 


\section{Projective spaces}\label{subsec:projective spaces} 
In this section, we consider a higher dimensional example in the projective spaces $\mathbb{P}_{\mathbb{C}}^{n-1}$ which is defined and discussed in Section~\ref{grass manifold} and Appendix B.2. In particular, we consider regression problems on $\mathbb{P}_{\mathbb{C}}^{4}$ which is a $8$-dimensional manifold. A similar regression function is considered as in \eqref{regression eqn}. Namely, we consider the function 
\[
f(X) = X^{\ast} \mathcal{M} X + \epsilon,
\] 
where $X$ is a $5$-dimensional unit norm complex column vector, $\mathcal M$ is a fixed  randomly generated positive definite hermitian matrix, $\epsilon$ is a randomly generated error and $X^{\ast}$ is the transpose conjugate of $X$. 

We notice that the domain of $f$ is diffeomorphic to the $9$-dimensional unit sphere $\mathbb{S}^9$ via the map:
\begin{align*}
\iota: \{X\in \mathbb{C}^5: \lVert X \rVert =1 \} &\to \mathbb{S}^9 \subseteq \mathbb{R}^{10} \\
(z_1,\dots, z_5)^\tp &\mapsto (\Re(z_1),\Im(z_1),\dots, \Re(z_5),\Im(z_5)).
\end{align*}
Using the embedding $\iota$, one can easily estimate $f$ by the extrinsic Gaussian process. However, after a little thought, one should be able to notice that $f$ is invariant under the rescaling by unit complex numbers, i.e., 
\[
f(\lambda X) = f(X),\quad \lambda\in \mathbb{C},\lvert \lambda \rvert = 1.
\] 
Therefore, $f$ descends to $\overline{f}: \mathbb{P}^4_{\mathbb{C}} \to \mathbb{R}$ defined by $\overline{f}([X]) = f(X)$ where $[X]$ is the point in $\mathbb{P}^4_{\mathbb{C}}$ represented by $X\in \mathbb{C}^5$. We remind the reader that $\mathbb{P}^4_{\mathbb{C}}$ can be identified with $\mathbb{S}^9/ \mathbb{S}^1$. With the above observation, the regression of $f$ is equivalent to the regression of $\overline{f}$ on $\mathbb{P}^4_{\mathbb{C}}$, for which our SiGP approach applies. {The training and testing datasets are generated by randomly picking $X_i \in \mathbb{P}^4$ and the corresponding $Y_i$ is computed by adding noises with different signals to noise ratios ($10$db, $20$db, $30$db). The comparison results of the root mean squared error (RMSE) of the prediction using SiGP approach and the naive GP using RBF kernel with euclidean distances in $\R^{10}$ are shown in table \ref{tab:p4}. Values in brackets are the standard deviations of RMSE. It is clear that the intrinsic GP performs significantly better than extrinsic methods in all three scenarios.
}

\begin{table}[t]
\caption{The comparison of the RMSE of predictive means of two methods on projective spaces.  Values in parentheses show the standard deviation.}
\label{tab:p4}
\vskip 0.15in
\begin{center}
\begin{small}
\begin{sc}
\begin{tabular}{lcccr}
\hline
      & naive GP RBF kernel & SiGP \\
\hline
30db  & 6.47(0.03)&  2.16 (0.05)  \\
20db &  6.5(0.14)&  2.13 (0.21)  \\
10db & 6.6(0.28)&  2.66 (0.28)   \\ 
\hline
\end{tabular}
\end{sc}
\end{small}
\end{center}
\vskip -0.1in
\end{table}

\section{Gorilla classifier}\label{appendix:gorilla}
Mathematically, the $k$-th image is represented by a vector $x_k = (a_1,b_1,\dots, a_4,b_4) \in \mathbb{R}^8, k =1,\dots, 59$, where for each $j=1,\dots, 4$, $(a_j,b_j)$ is the coordinate of the $j$-th landmark on the image. 
We represent the $\mathbb{R}^2$ landmark coordinates by $\mathbb{C}$ so that $x_k = (a_1 + i b_1,\dots, a_4 + ib_4) \in \mathbb{C}^4$.
We can get a translation and rotation invariant feature representation $[\overline{x}_k], k =1,\dots, 59$, where $[\overline{x}_k]$ is a point in $\mathbb{P}^2_{\mathbb{C}} = \left( \mathbb{C}^3\setminus \{0\} \right)/ (\mathbb{C}\setminus \{0\})$~\footnote{Note that the modelling of our classification problem on $\mathbb{P}^2_{\mathbb{C}}$ follows the same construction as in \cite{kendall1977diffusion}.} and
\[
\overline{x}_k = (z_2-z_1,z_3-z_1,z_4- z_1) \in \mathbb{C}^3 \setminus \{0\}.
\]

\section{Diffusion tensor image classification results for all sites}\label{appendix:dti}

We notice that all the positive definite matrices in the data set are contained in 
\[
M = \{
A\in \mathbb{S}^3_{++}: \lambda_i (A) > \lambda_j(A), 1\le i<j \le 3, v_1 \not\in \mathbb{E}
\},
\] 
where positive numbers $\lambda_1(A),\lambda_2(A),\lambda_3(A)$ are decreasingly ordered eigenvalues of $A$, $v_1\in \mathbb{R}^3$ is the unit eigenvector corresponding to $\lambda_1(A)$ and $\mathbb{E}$ is the subspace of $\mathbb{R}^3$ spanned by $e_2 = (0,1,0)^\tp$ and $e_3 = (0,0,1)^\tp$. 
Then, we observe that by eigenvalue decomposition, we may identify $M$ with $X \times \mathbb{R}_>^3$ where $X$ is the open subset of $\SO(3)$ consisting of $3\times 3$ orthogonal matrices whose first column is not contained in $\mathbb{E}$ and $\mathbb{R}_{+,>}^3$ is the open subset of $\mathbb{R}^3$ consisting of increasingly ordered positive real numbers. To be more precise, we have 
\begin{align*}
M &\xrightarrow{\simeq} X \times \mathbb{R}_+^3 , \ \  
A \mapsto (P, \lambda_1(A),\lambda_2(A),\lambda_3(A)).
\end{align*}
Here $A = P \begin{bmatrix}
\lambda_1(A) & 0 & 0 \\
0 & \lambda_2(A) & 0\\
0 & 0 &\lambda_3(A)
\end{bmatrix} P^\tp$ is the uniquely determined  eigenvalue decomposition of $A\in M$. Furthermore, since $X $ can be identified with $(\mathbb{S}^2 \setminus \{0\}) \times \mathbb{S}^1$ via the map 
\begin{align*}
X  &\hookrightarrow \mathbb{S}^2 \times \mathbb{S}^1 , \ 
P= (v_1,v_2,v_3)\mapsto (v_1,\theta), 
\end{align*}
where $\theta$ is the uniquely determined angle such that 
\[
\begin{bmatrix}
v'_2 & v'_3
\end{bmatrix} \begin{bmatrix}
\cos (\theta) & \sin (\theta) \\
-\sin (\theta) & \cos (\theta)
\end{bmatrix} = \begin{bmatrix}
v_2 & v_3
\end{bmatrix}
\]
and $\begin{bmatrix}
v_1 & v'_2 & v'_3
\end{bmatrix} \in \SO(3)$ is the orthogonal factor in the QR factorization of the matrix $\begin{bmatrix}
v_1 & e_2 & e_3
\end{bmatrix}$. Geometrically, both $\{v'_2,v'_3\}$ and $\{v_2,v_3\}$ are orthonormal frames of the tangent space $T_{v_1} \mathbb{S}^2$ and $\theta$ is the angle required to rotate $\{v'_2,v'_3\}$ to $\{v_2,v_3\}$. With above maps, those $3\times 3$ positive definite matrices in the data set correspond to points in $\mathbb{S}^2 \times \mathbb{S}^1 \times \mathbb{R}_{+,>}^3$, which is the underlying manifold we use to build our binary classifiers.

Now our goal can be achieved by building the binary SiGP classifier for given $46$ points in the product manifold $\mathbb{S}^2 \times \mathbb{S}^1 \times \mathbb{R}^3_+$ and repeat this process for the prespecified $75$ sites in the corpus callosum. We define $y_i \in \{ 0 ,1 \}$ where 0 represents the healthy subject and 1 for HIV+. Let $s_k$ be a point in $\mathbb{S}^2 \times \mathbb{S}^1 \times \mathbb{R}^3_{+,>}$. Similar to \eqref{eq:gorilla}, the heat kernel $p^{ \mathbb{S}^2 \times \mathbb{S}^1 \times \mathbb{R}^3_{+,>}  }$ of the product manifold $ \mathbb{S}^2 \times \mathbb{S}^1 \times \mathbb{R}^3_{+,>}$ is used for constructing the intrinsic GP classifier. We remark that by Proposition~\ref{prop:heat kernel product manifold}, $p^{ \mathbb{S}^2 \times \mathbb{S}^1 \times \mathbb{R}^3_{+,>}  }$ is simply the product of $p^{\mathbb{S}^2}$ the kernel of $\mathbb{S}^2$, $p^{ \mathbb{S}^1}$ the kernel of $ \mathbb{S}^1$ and $p^{\mathbb{R}^3}$ the kernel of $\mathbb{R}^3$, all of which can easily be estimated by Algorithm~\ref{alg:acc estimation of heat kernel}.

\begin{table}[t]
\caption{AUC-ROC scores of the diffusion tensor images classification. Values in parentheses show the standard deviation. For a naive classifier with $P(y_i=0)=P(y_i=1)=0.5$, the corresponding  roc-aus score is 0.5. }  
\label{tab:DTIroc}
\vskip 0.15in
\begin{center}
\begin{small}
\begin{sc}
\begin{tabular}{lccccccccr}
\hline
    & SiGP & naive GP  &SVM  & LR \\
\hline
site13  & 0.83(0.12)& 0.60(0.18) & 0.46(0.29) & 0.59(0.20)\\
site55 &  0.83(0.14) & 0.44(0.14) & 0.44(0.15) &0.50(0.19) \\ 
site29 &  0.83(0.08) &  0.56(0.15)&0.44(0.18)&0.63(0.19)\\
site39 & 0.82(0.14) &0.57(0.25) &0.53(0.20)&0.52(0.25)\\
site50 &  0.87(0.08) &0.48(0.21)& 0.53(0.23)&0.54(0.17)\\
\hline
\end{tabular}
\end{sc}
\end{small}
\end{center}
\vskip -0.1in
\end{table}

\begin{table}[ht]
\caption{Diffusion tensor images classification for all sites. Values in parentheses show the standard deviation  }
\label{tab:DTIallarc}
\vskip 0.15in
\begin{center}
\begin{small}
\begin{sc}
\begin{tabular}{lccccccccr}
\hline
      & SiGP  & naive GP & SVM &logistic reg\\
\hline
Cross entropy  & 0.62(0.06)& 0.69(0.02) &0.64(0.11)&0.65(0.03)  \\
AUC-ROC score&  0.61(0.16) & 0.53(0.19) &0.50(0.20)&0.55(0.21)  \\
\hline
\end{tabular}
\end{sc}
\end{small}
\end{center}
\vskip -0.1in
\end{table}

The classification results for all the $75$ sites are summarised as the average cross entropy and AUC-ROC scores in Table~\ref{tab:DTIallarc} which might be of independent interest. It is worthy to notice that in this comparison, the outperformance of SiGP against other approaches, though it is still visible, is not as significant as that for sensitive sites. This phenomenon is caused by the fact that most of the $75$ sites are not sensitive to HIV\footnote{In fact, determining sensitive sites is exactly the purpose of the DTI technique.} and hence our SiGP classifier tends to a naive classifier when we take the average across all the sites.

\end{document}